\documentclass[reqno]{amsart}
\pdfoutput=1

\usepackage{thmtools}
\usepackage{amssymb}
\usepackage[hyphens]{url}
\usepackage[breaklinks=true,unicode]{hyperref}
\usepackage{amsmath} 
\usepackage{mathtools} 
\usepackage{todonotes}
\usepackage{tikz-cd}
\usepackage{mathrsfs} 
\usepackage{stmaryrd} 
\usepackage{cleveref}
\usepackage{enumitem}
\usepackage[top=3.5cm,bottom=3.5cm,left=4cm,right=4cm]{geometry}

\usepackage{scalerel}
\usepackage{stackengine}
\newlength\arrowheight

\newcommand\doubledownarrow{%
  \mathrel{\ThisStyle{%
    \setlength{\arrowheight}{\heightof{$\SavedStyle\downarrow$}}%
    \scalerel*{\stackengine{.3\arrowheight}{$\SavedStyle\downarrow$}%
      {$\SavedStyle\downarrow$}{O}{c}{F}{F}{L}}{\downarrow}}
}}

\expandafter\def\expandafter\UrlBreaks\expandafter{\UrlBreaks
  \do\a\do\b\do\c\do\d\do\e\do\f\do\g\do\h\do\i\do\j%
  \do\k\do\l\do\m\do\n\do\o\do\p\do\q\do\r\do\s\do\t%
  \do\u\do\v\do\w\do\x\do\y\do\z\do\A\do\B\do\C\do\D%
  \do\E\do\F\do\G\do\H\do\I\do\J\do\K\do\L\do\M\do\N%
  \do\O\do\P\do\Q\do\R\do\S\do\T\do\U\do\V\do\W\do\X%
  \do\Y\do\Z}

\hypersetup{
  linkcolor  = black, 
  citecolor  = magenta!90!black, 
  urlcolor   = black, 
  colorlinks = true,
}

\mathchardef\mhyphen="2D 
\newcommand{\op}{\mathrm{op}}

\renewcommand{\phi}{\varphi}
\renewcommand{\theta}{\vartheta}
\renewcommand{\epsilon}{\varepsilon}

\newcommand{\1}{\mathbf{1}} 

\newcommand{\D}{D} 

\renewcommand{\inf}{\bigwedge}
\renewcommand{\sup}{\bigvee}

\newcommand{\dd}{{\doubledownarrow}\,}

\renewcommand{\d}{{\downarrow}\,}
\renewcommand{\u}{{\uparrow}\,}

\newcommand{\repi}{\twoheadrightarrow}
\newcommand{\emb}{\rightarrowtail}

\newcommand{\cat}[1]{\mathsf{#1}}

\newcommand{\Set}{\cat{Set}}

\newcommand{\CRing}{\cat{CRing}}
\newcommand{\CompOrd}{\cat{CompOrd}}
\DeclareMathOperator{\RegEpi}{RegEpi}

\DeclareMathOperator{\KSh}{Sh_{\mathcal{K}}} 
\DeclareMathOperator{\OSh}{Sh_{\Omega}} 
\newcommand{\KShr}[1]{\mathrm{Sh}^{#1}_{\mathcal{K}}} 
\newcommand{\OmShr}[1]{\mathrm{Sh}^{#1}_{\Omega}} 
\newcommand{\sKShr}[1]{\mathrm{s\mhyphen Sh}^{#1}_{\mathcal{K}}} 
\newcommand{\sOmShr}[1]{\mathrm{s\mhyphen Sh}^{#1}_{\Omega}} 

\DeclareMathOperator{\Lan}{Lan}
\DeclareMathOperator{\Ran}{Ran}

\newcommand{\ksat}{\mathcal{K}}
\DeclareMathOperator{\Equiv}{Equiv}
\DeclareMathOperator{\Sub}{Sub}
\DeclareMathOperator{\Quo}{Quo}

\DeclareMathOperator*{\colim}{colim}
\DeclareMathOperator{\coeq}{coeq}

\newcommand{\ros}[3]{{#1}_{{#2},{#3}}} 

\DeclareMathOperator{\ID}{Id}
\DeclareMathOperator{\RID}{RId}
\DeclareMathOperator{\JRID}{JRId}
\DeclareMathOperator{\Max}{Max}

\newcommand{\Filt}{\mathrm{Filt}}
\newcommand{\Scott}{\sigma\mathrm{Filt}}
\newcommand{\F}{\mathcal{F}}

\newcommand{\K}{\mathcal{K}}

\declaretheorem[name=Theorem, refname={Theorem,Theorems}, Refname={Theorem,Theorems}, numberwithin=section]{theorem}

\declaretheorem[name=Theorem, numbered = no]{theorem*}

\declaretheorem[name=Proposition, refname={Proposition,Propositions}, Refname={Proposition,Propositions}, sibling=theorem]{proposition}

\declaretheorem[name=Proposition, numbered = no]{proposition*}

\declaretheorem[name=Lemma, refname={Lemma,Lemmas}, Refname={Lemma,Lemmas}, sibling=theorem]{lemma}

\declaretheorem[name=Lemma, numbered = no]{lemma*}

\declaretheorem[name=Corollary, refname={Corollary,Corollaries}, Refname={Corollary,Corollaries}, sibling=theorem]{corollary}

\declaretheorem[name=Corollary, numbered = no]{corollary*}

\declaretheorem[name=Definition, refname={Definition,Definitions}, Refname={Definition,Definitions}, sibling=theorem, style=definition]{definition}

\declaretheorem[name=Example, refname={Example,Examples}, Refname={Example,Examples}, sibling=theorem, style=definition]{example}

\declaretheorem[name=Remark, refname={Remark,Remarks}, Refname={Remark,Remarks}, sibling=theorem, style=remark]{remark}

\declaretheorem[name=Remark, style=remark, numbered=no]{remark*}

\declaretheorem[name=Claim, numbered=no, style=remark]{claim*}

\crefname{section}{Section}{Sections}
\Crefname{section}{Section}{Sections}

\crefname{axiom}{Axiom}{Axioms}
\Crefname{axiom}{Axiom}{Axioms}

\crefname{item}{Item}{Items}
\Crefname{item}{Item}{Items}

\title{Barr-Exact Categories and Soft Sheaf Representations}

\author{Marco Abbadini}
\address{Dipartimento di Matematica, Universit\`{a} degli Studi di Salerno, Italy}\email{mabbadini@unisa.it}

\author{Luca Reggio}
\address{Department of Computer Science, University College London, United Kingdom}
\email{l.reggio@ucl.ac.uk}

\thanks{Research supported by the European Union's Horizon 2020 research and innovation programme under the Marie Sk{\l}odowska-Curie grant agreement No 837724 and by the Italian Ministry of University and Research through the PRIN project n.\ 20173WKCM5 \emph{Theory and applications of resource sensitive logics}.}

\begin{document}

\begin{abstract}
It has long been known that a key ingredient for a sheaf representation of a universal algebra $A$ consists in a distributive lattice of commuting congruences on~$A$. The sheaf representations of universal algebras (over stably compact spaces) that arise in this manner have been recently characterised by Gehrke and van Gool (J.~Pure Appl.~Algebra, 2018), who identified the central role of the notion of \emph{softness}. 

In this paper, we extend the scope of this theory by replacing varieties of algebras with Barr-exact categories, thus encompassing a number of ``non-algebraic'' examples. Our approach is based on the notion of \emph{$\ksat$-sheaf}: intuitively, whereas sheaves are defined on open subsets, $\ksat$-sheaves are defined on compact ones. Throughout, we consider sheaves on complete lattices rather than spaces; this allows us to obtain point-free versions of sheaf representations whereby spaces are replaced with frames.

These results are used to construct sheaf representations for the dual of the category of compact ordered spaces, and to recover Banaschewski and Vermeulen's point-free sheaf representation of commutative Gelfand rings (Quaest.~Math., 2011).
\end{abstract}

\maketitle

\section{Introduction}
Sheaf representations of universal algebras have been investigated since the 1970s, see e.g.\ \cite{Comer1971,Cornish1977,Davey1973,Keimel1970,Wolf1974}, inspired by several results for rings and modules obtained in the 1960s, see e.g.\ \cite{DH1966,DH1968,EGA1,Pierce1967}. 
In particular it was observed that, for a universal algebra~$A$, any distributive lattice of pairwise commuting congruences on $A$ induces a sheaf representation of $A$, i.e.\ a sheaf whose algebra of global sections is isomorphic to~$A$~\cite{Wolf1974}.
The sheaf representations over \emph{stably compact} spaces~\cite{Lawson2011} arising in this way were characterised by Gehrke and van Gool~\cite{GehrkeGool2018}, who recognised the key role of the notion of softness---which originated with Godement's treatment of homological algebra~\cite{Godement1958}. A sheaf over a space $X$ is \emph{soft} if, for all compact saturated\footnote{A subset of a topological space is \emph{saturated} if it is an intersection of open sets. In a $T_1$ space, all subsets are saturated.} subsets $K\subseteq X$, every (continuous) local section over $K$ can be extended to a (continuous) global section. In~\cite{GehrkeGool2018}, a bijection was established between isomorphism classes of soft sheaf representations of an algebra $A$ over a stably compact space $X$, and frame homomorphisms from the \emph{co-compact dual} frame of $X$ to a frame of pairwise commuting congruences on $A$.

A sheaf representation of an algebra can be regarded as a generalisation of a representation in terms of continuous functions. For example, the Gelfand--Naimark theorem~\cite{GN1943} states that for every commutative unital $\mathrm{C}^*$-algebra $A$ there is an isomorphism
\[
A\cong \mathrm{C}(\Max{A},\mathbb{C})
\]
where $\mathrm{C}(\Max{A},\mathbb{C})$ is the $\mathrm{C}^*$-algebra of all continuous complex-valued functions on the maximal spectrum $\Max{A}$ of $A$. More precisely, it provides a characterisation of the image of the embedding
\[
A \rightarrowtail \prod_{\mathfrak{m}\in\Max{A}}{A/{\mathfrak{m}}}, \ \ \ a\mapsto (a/{\mathfrak{m}})_{\mathfrak{m}\in\Max{A}}
\]
where each $A/{\mathfrak{m}}$ is an isomorphic copy of $\mathbb{C}$.
Sheaf representations extend these ideas to a wider class of rings---and, more generally, universal algebras---by allowing the factors in the direct product to vary ``continuously''.

Thus, disregarding the topological constraint of continuity of global sections, sheaf representations of universal algebras are akin to embeddings into direct products. From this standpoint, the softness condition for sheaves is related to a basic concept of universal algebra, namely that of \emph{subdirect representation}. In fact, any sheaf representation $F$ of a universal algebra $A$ on a space $X$ induces an embedding
\[
\nu\colon A \emb \prod_{x\in X}{F_x}
\]
where $F_x$ is the stalk of $F$ at $x$. If $F$ is soft then $\nu$ is a \emph{subdirect embedding}, i.e.\ for all $y\in X$ the composition of $\nu$ with the product projection $\prod_{x\in X}{F_x}\repi F_{y}$ is surjective.

In this article, we generalise Gehrke and van Gool's characterisation of soft sheaf representations by replacing varieties of finitary algebras---in which sheaves take values---with any \emph{Barr-exact} category.
Barr-exact categories, introduced in~\cite{BarrGrilletOsdol1971}, are a non-additive generalisation of Abelian categories. Examples of Barr-exact categories include most ``algebraic-like'' categories such as varieties of (possibly infinitary) algebras, any topos, the category of compact Hausdorff spaces and its opposite category.

This allows us to construct soft sheaf representations of all objects in the category $\CompOrd^\op$ opposite to the category of compact ordered spaces and continuous monotone maps. This category can be regarded as an extension of the variety $\cat{DLat}$ of bounded distributive lattices, in the sense that $\CompOrd^\op$ admits a full subcategory equivalent to $\cat{DLat}$ (this follows from Priestley duality between bounded distributive lattices and totally order-disconnected compact ordered spaces~\cite{Priestley1970}). The category $\CompOrd^\op$ is Barr-exact and even equivalent to a variety of algebras, but not a finitary one.

We hasten to point out that, while the intended application of our results concerns Barr-exact categories, we develop the theory more generally for \emph{regular} categories~\cite{BarrGrilletOsdol1971}. Examples of regular categories that are not Barr-exact include quasi-varieties of (possibly infinitary) algebras and the category of Boolean (i.e., compact, Hausdorff and zero-dimensional) spaces and continuous maps. 

Barring some examples and applications, we always work with sheaves over complete lattices. The usual notion of a sheaf on a space $X$ is recovered by considering sheaves over the frame of opens of $X$, but this approach accommodates also point-free sheaf representations by taking sheaves on possibly non-spatial frames. 
For example, we illustrate how Banaschewski and Vermeulen's sheaf representation of commutative Gelfand rings on compact regular frames~\cite{BV2011} can be recovered as a special case of our results.

Our approach crucially relies on the notion of \emph{$\ksat$-sheaf}. In the spatial setting, a $\ksat$-sheaf can be thought of as a ``sheaf defined on compact saturated subsets'' (or equivalently, in the case of $T_1$ spaces, on compact subsets) instead of open ones. The concept of $\ksat$-sheaf essentially originates with Leray's pioneering work on sheaves and sheaf cohomology~\cite{Leray1945} (see also~\cite[Chapter IV, \S 7.B]{Dieudonne2009}) and has been fruitfully employed by Lurie in the theory of $\infty$-categories, cf.~\cite[\S 7.3.4]{Lurie2009}.

\subsection*{Outline}
In \cref{s:prelim}, we recall the basic definitions and properties pertaining to the theory of regular and Barr-exact categories. $\K$-sheaves over a complete lattice, with values in a regular category, are introduced in \cref{K-sheaves}. In \cref{s:soft-K-sheaf-repr} the notion of softness for $\K$-sheaves is defined and, for any object of a regular category, an isomorphism is established between a category of soft $\ksat$-sheaf representations and a category of monotone maps preserving finite infima and arbitrary suprema (\cref{th:equiv-soft-sheaf-repr}). 

\cref{s:domains} contains some background material on domains and algebraic lattices that is needed in \cref{s:K-vs-Omega} to show that, under appropriate assumptions, $\K$-sheaves are equivalent to ordinary sheaves (\cref{equiv-Om-K-sheaves}). This equivalence is then extended to (soft) sheaf representations and leads to our main result, \cref{t:gen}. Finally, in \cref{s:examples}, these results are applied to study soft sheaf representations of commutative Gelfand rings and of objects in the dual of the category of compact ordered spaces.

\subsection*{Notation and terminology}
A poset $P$ is \emph{directed} if each of its finite subsets has an upper bound; equivalently, if it is non-empty and any two of its elements admit an upper bound. A subset $\D\subseteq P$ is said to be directed if it satisfies the previous condition with respect to the order induced by $P$. The order-dual notion is that of \emph{codirected} subset.

We often identify a preordered set $S$ with the (small) category whose set of objects is $S$ and such that, for all $s,t\in S$, there is exactly one morphism $s\to t$ if $s\leq t$, and there is no morphism otherwise. More generally, categories in which there is at most one morphism between any two objects are identified with (possibly large) preorders. 

Arrows $\emb$ and $\repi$ denote, respectively, monomorphisms and regular epimorphisms. The terminal object in a category, if it exists, is denoted by $\1$.

Whenever $\cat{C}$ is a category, $\cat{C}^\op$ denotes the \emph{opposite category} obtained by reversing the direction of arrows in $\cat{C}$. This applies in particular when $\cat{C}$ is a poset, e.g.\ the complete lattice $\Omega(X)$ of open subsets of a topological space $X$, ordered by set-theoretic inclusion.

We write $\Omega$ for the contravariant functor from the category of topological spaces and continuous maps, to the category of frames and their homomorphisms, that sends a continuous map $f\colon X\to Y$ to the frame homomorphism $\Omega(f)\coloneqq f^{-1}\colon \Omega(Y)\to\Omega(X)$.

A \emph{$\cat{C}$-valued presheaf} on a poset $P$ is a functor
$
F \colon P^\op \to \cat{C}.
$ 
If $p,q\in P$ satisfy $p\leq q$, the image under $F$ of the unique arrow $q\to p$ in $P^\op$ is denoted by $\ros{F}{q}{p}\colon F(q)\to F(p)$.

\section{Preliminaries on Regular Categories}\label{s:prelim}
We recall some basic definitions and facts concerning regular and Barr-exact categories.

\subsection{Subobjects and quotients} Consider an arbitrary category $\cat{C}$ and an object $A$ of~$\cat{C}$. The collection of all monomorphisms in $\cat{C}$ with codomain $A$ carries a natural preorder~$\leq$ defined as follows. Given monomorphisms $m\colon S\emb A$ and $n\colon T\emb A$, set
\[
m \leq n \ \ \Longleftrightarrow \ \ \exists \, l. \ m=n\circ l \hspace{2em}
\begin{tikzcd}
T \arrow[rightarrowtail]{r}{n} & A \\
S \arrow[rightarrowtail]{ur}[swap]{m} \arrow[dashed]{u}{l} & {}	
\end{tikzcd}
\]
(Note that, if it exists, such an $l$ is a monomorphism.) The symmetrization $\sim$ of the preorder $\leq$ can be characterised explicitly as follows: $m\sim n$ if, and only if, there exists an isomorphism $l$ such that $m=n\circ l$. A \emph{subobject} of $A$ is a $\sim$-equivalence class of monomorphisms with codomain $A$, and the collection of all subobjects of $A$ is denoted by $\Sub{A}$. The preorder $\leq$ induces a partial order on $\Sub{A}$, which we denote again by $\leq$. As $A$ may admit a proper class of subobjects, in general $\Sub{A}$ is a large poset.

In the same fashion, we can define the (large) poset of quotients of an object $A$. To this end, we introduce a preorder on the class of all regular epimorphisms with domain $A$. We use again the symbol $\leq$ for this preorder; it will be clear from the context to which (pre)order we are referring. Given regular epimorphisms $f\colon A \repi B$ and $g\colon A \repi C$, set
\[
f \leq g \ \ \Longleftrightarrow \ \ \exists \, h. \ f=h\circ g \hspace{2em}
\begin{tikzcd}
A \arrow[two heads]{r}{g}\arrow[swap, two heads]{rd}{f}	& C \arrow[dashed]{d}{h}\\
& B
\end{tikzcd}
\]
(Note that, if it exists, such an $h$ is an epimorphism, but need not be a regular epimorphism unless $\cat{C}$ admits (regular epi, mono) factorisations.)
As before, the symmetrization~$\sim$ of the preorder $\leq$ can be characterised explicitly by: $f\sim g$ if, and only if, there exists an isomorphism $h$ such that $f=h\circ g$. A \emph{quotient} of $A$ is a $\sim$-equivalence class of regular epimorphisms with domain $A$, and the collection of all quotients of $A$ is denoted by $\Quo{A}$. The preorder $\leq$ induces a partial order on $\Quo{A}$, that we shall denote again by~$\leq$. Equivalently, $\Quo{A}$ can be identified with the poset of \emph{regular} subobjects of $A$ in the opposite category $\cat{C}^\op$.

We will often work with the category \[\RegEpi{A}\] whose objects are regular epimorphisms in $\cat{C}$ with domain $A$. For any two regular epimorphisms $f\colon A \repi B$ and $g\colon A \repi C$, an arrow $f\to g$ in $\RegEpi{A}$ is an arrow $h\colon B\to C$ in $\cat{C}$ such that $h\circ f = g$. In other words, $\RegEpi{A}$ is a full subcategory of the coslice category $A/{\cat{C}}$. Note that, since each object of $\RegEpi{A}$ is an epimorphism in $\cat{C}$, the category $\RegEpi{A}$ is a (large) preorder.
The poset reflection of $\RegEpi{A}$ coincides with the \emph{opposite} of $\Quo{A}$.

\subsection{Regular categories}

\begin{definition}
A category $\cat{C}$ is \emph{regular} if it satisfies the following conditions:
\begin{enumerate}[label=(\roman*)]
\item $\cat{C}$ has finite limits.
\item $\cat{C}$ has (regular epi, mono) factorisations, i.e.\ every arrow $f$ in $\cat{C}$ can be written as $f=m\circ e$ where $e$ is a regular epimorphism and $m$ a monomorphism.
\item Regular epimorphisms in $\cat{C}$ are stable under pullbacks along any morphism.
\end{enumerate}
\end{definition}

Note that, because every regular epimorphism is a strong epimorphism, for every commutative square
\[\begin{tikzcd}
A \arrow[twoheadrightarrow]{r} \arrow{d} & B \arrow{d} \\
C \arrow[rightarrowtail]{r} & D 
\end{tikzcd}\]
there exists a (unique) \emph{diagonal filler}, i.e.\ an arrow $B\to C$ making the ensuing triangles commute. Although we shall not need this fact, let us mention that in any regular category the strong and regular epimorphisms coincide; see e.g.\ \cite[Proposition~1.4, p.~129]{BarrGrilletOsdol1971}.

Some useful consequences of the axioms for a regular category are collected in the following lemma; these properties rely on the fact that the (regular epi, mono) factorisation system in a regular category is orthogonal, proper and stable. Cf.\ e.g.~\cite{freyd1972categories} or \cite[\S1.5]{FreydScedrov1990}.
\begin{lemma}\label{regular-properties}
The following statements hold in any regular category $\cat{C}$:
\begin{enumerate}[label=(\alph*)]
\item\label{composition} The composition of regular epimorphisms is again a regular epimorphism.
\item\label{cancellation} If $f\circ g$ is a regular epimorphism, then so is $f$.
\item\label{pullback-pushout} Any pullback square consisting entirely of regular epimorphisms is also a pushout square.
\end{enumerate}
\end{lemma}
\begin{proof}
For~\cref{composition,cancellation}, see~\cite[Propositions~1.10 and~1.11, p.~133]{BarrGrilletOsdol1971}. \Cref{pullback-pushout} follows from (the dual of) the main result of \cite{Ringel1972} (cf.\ also \cite[\S1.565]{FreydScedrov1990} or \cite[Remark~5.3]{CKP1993}).
\end{proof}

Whenever $A$ is an object of a regular category $\cat{C}$, the preorder $\RegEpi{A}$ admits finite infima. Just observe that an infimum of a finite set of regular epimorphisms 
\[\{f_i\colon A\repi B_i\mid i\in I\}\]
is given by the (regular epi, mono) factorisation of the induced morphism $A\to \prod_{i\in I}{B_i}$.

$\RegEpi{A}$ has also a minimum, namely the identity of $A$; in fact, $A\repi B$ is a minimum in $\RegEpi{A}$ if, and only if, it is an isomorphism. However, non-empty suprema in $\RegEpi{A}$ may fail to exist. For the next lemma, recall that the pushout of a regular epimorphism---if it exists---is again a regular epimorphism.

\begin{lemma}\label{l:binary-infima}
Let $\cat{C}$ be a regular category, $A$ an object of $\cat{C}$, and $f \colon A \repi B$ and $g \colon A \repi C$ regular epimorphisms. Then $f$ and $g$ admit a supremum in $\RegEpi{A}$ if, and only if, they admit a pushout in $\cat{C}$.
\begin{equation}\label{eq:pushout-inf}
\begin{tikzcd}
A \arrow[twoheadrightarrow]{r}{f} \arrow[twoheadrightarrow]{d}[swap]{g} & B \arrow[twoheadrightarrow, dashed]{d}{\eta_1} \\
C \arrow[twoheadrightarrow, dashed]{r}{\eta_2} & H \arrow[ul, phantom, "\ulcorner", very near start]
\end{tikzcd}
\end{equation}
In that case, the composite $\eta_1 \circ f = \eta_2 \circ g$ is a supremum of $f$ and $g$.
\end{lemma}
\begin{proof}
	Suppose that the diagram in~\cref{eq:pushout-inf} is a pushout.
	By~\cref{composition} in~\cref{regular-properties}, $h\coloneqq \eta_1 \circ f$ ($= \eta_2 \circ g$) is an element of $\RegEpi{A}$ that is above $f$ and $g$, and the universal property of the pushout readily implies that $h$ is a supremum of $f$ and $g$. 
	
	For the converse direction, assume $f$ and $g$ admit a supremum $h\colon A\repi H$ in $\RegEpi{A}$. In particular, as $h$ is above $f$ and $g$, there exist morphisms $\eta_1\colon B\to H$ and $\eta_2\colon C\to H$ such that $\eta_1\circ f=h=\eta_2\circ g$. Hence the following square commutes:
	\[\begin{tikzcd}
		A \arrow[twoheadrightarrow]{r}{f} \arrow[twoheadrightarrow]{d}[swap]{g} & B \arrow{d}{\eta_1} \\
		C \arrow{r}{\eta_2} & H
	\end{tikzcd}\]
	Note that $\eta_1,\eta_2$ are regular epimorphisms by~\cref{cancellation} in~\cref{regular-properties}. We claim that the square above is a pushout. Consider morphisms $\sigma_1\colon B\to J$ and $\sigma_2\colon C\to J$ such that $\sigma_1\circ f=\sigma_2\circ g$, and let $(e,m)$ be the (regular epi, mono) factorisation of $\sigma_1\circ f=\sigma_2\circ g$. Then there are diagonal fillers $\tau_1$ and $\tau_2$ as displayed below.
	\begin{center}
		\begin{tikzcd}
			A \arrow[twoheadrightarrow]{r}{f} \arrow[twoheadrightarrow]{d}[swap]{e} & B \arrow[dashed]{dl}[description]{\tau_1} \arrow{d}{\sigma_1} \\
			K \arrow[rightarrowtail]{r}{m} & J
		\end{tikzcd}
		\ \ \ \ \ \ \ \ 
		\begin{tikzcd}
			A \arrow[twoheadrightarrow]{r}{g} \arrow[twoheadrightarrow]{d}[swap]{e} & C \arrow{d}{\sigma_2} \arrow[dashed]{dl}[description]{\tau_2}  \\
			K \arrow[rightarrowtail]{r}{m} & J
		\end{tikzcd}
	\end{center}
	Hence, $e$ is above $f$ and $g$ in $\RegEpi{A}$. Since $h$ is a supremum of $f$ and $g$, there exists a morphism $\xi\colon H\to K$ satisfying $e= \xi\circ h$. Because $f$ and $g$ are epimorphisms, we see that the following diagram commutes.
	\[\begin{tikzcd}
		A \arrow[twoheadrightarrow]{r}{f} \arrow[twoheadrightarrow]{d}[swap]{g} & B \arrow[twoheadrightarrow]{d}{\eta_1} \arrow[bend left = 30, looseness=1, two heads]{ddr}[description]{\tau_1} & \\
		C \arrow[twoheadrightarrow]{r}{\eta_2} \arrow[bend right = 30, looseness=1, two heads]{drr}[description]{\tau_2} & H \arrow{dr}[description]{\xi} & \\
		& & K
	\end{tikzcd}\]
	The composite morphism $m\circ \xi\colon H\to J$ then satisfies $(m\circ \xi)\circ \eta_1=\sigma_1$ and $(m\circ \xi)\circ \eta_2=\sigma_2$.  As $\eta_1$ (or, equivalently, $\eta_2$) is an epimorphism, it follows that $m\circ \xi$ is the unique arrow with this property.
\end{proof}
\begin{remark}
	With regards to the previous lemma, a closely related fact was proved by Burgess and Caicedo, cf.~\cite[Proposition~10]{BurgessCaicedo1981}. 
\end{remark}
The following is a consequence of~\cref{l:binary-infima} and the preceding discussion:
\begin{corollary}\label{cor:QuoA-bounded-lattice}
	Let $\cat{C}$ be a regular category admitting pushouts. For any object $A$ of $\cat{C}$, its poset of quotients $\Quo{A}$ is a (possibly large) bounded lattice.
\end{corollary}

\subsection{Barr-exact categories}
Let $\cat{C}$ be a regular category and let $A$ be an object of $\cat{C}$. A subobject $\langle p_1,p_2\rangle \colon R\emb A\times A$ is called a \emph{relation} on $A$, and it is an \emph{equivalence relation} provided it satisfies the following three properties:
\begin{enumerate}[label=(\roman*)]
\item[] \emph{Reflexivity.} There exists an arrow $d\colon A \to R$ in $\cat{C}$ making the following diagram commute:
\[\begin{tikzcd}
A \arrow[dashed]{rr}{d} \arrow[rightarrowtail]{dr}[swap]{\langle 1_A,1_A\rangle} & & R \arrow[rightarrowtail]{dl}{\langle p_1,p_2\rangle} \\
{} & A\times A & {}
\end{tikzcd}\] 
\item[] \emph{Symmetry.} There exists an arrow $s\colon R \to R$ in $\cat{C}$ making the following diagram commute:
\[\begin{tikzcd}
R \arrow[dashed]{rr}{s} \arrow[rightarrowtail]{dr}[swap]{\langle p_2,p_1\rangle} & & R \arrow[rightarrowtail]{dl}{\langle p_1,p_2\rangle} \\
{} & A\times A & {}
\end{tikzcd}\] 
\item[] \emph{Transitivity.} For any pullback diagram in $\cat{C}$ as on the left-hand side below, there exists an arrow $t\colon P\to R$ such that the rightmost diagram commutes:
\[\begin{tikzcd}
P \arrow{r}{\pi_2} \arrow[swap]{d}{\pi_1} \arrow[dr, phantom, "\lrcorner", very near start]  & R \arrow{d}{p_1} \\
R \arrow{r}{p_2} & A
\end{tikzcd}
 \ \ \ \ \ \ 
\begin{tikzcd}
P\arrow[swap]{rd}{\langle p_1\circ \pi_1,p_2\circ\pi_2\rangle}\arrow[dashed]{rr}{t}&&R\arrow[tail]{dl}{\langle p_1,p_2\rangle}\\
& A\times A &
\end{tikzcd}
\]
\end{enumerate}

\begin{example}
If $\cat{C}$ is a variety of algebras then an equivalence relation on an algebra~$A$, in the sense above, coincides with the usual notion of \emph{congruence}.
\end{example}

With any arrow $f\colon A\to B$ in $\cat{C}$ we can associate a relation on $A$, known as the \emph{kernel pair} of $f$. This is obtained by taking the pullback of $f$ along itself:
\[\begin{tikzcd}
R \arrow{r}{p_1} \arrow{d}[swap]{p_2} \arrow[dr, phantom, "\lrcorner", very near start] & A \arrow{d}{f} \\
A \arrow{r}{f} & B 
\end{tikzcd}\]
The kernel pair $\ker{f}\coloneqq \langle p_1,p_2\rangle$ is a relation on $A$. Just recall that the pullback of $f$ along itself can equivalently be computed as the equaliser of the pair of parallel arrows $f\circ \pi_1, f\circ \pi_2 \colon A\times A \rightrightarrows B$, where $\pi_1,\pi_2\colon A\times A \rightrightarrows A$ are the product projections. In fact, $\ker{f}$ is always an equivalence relation on $A$. 

The collection $\Equiv{A}$ of all equivalence relations on $A$ carries a natural partial order, induced by the order of $\Sub{(A\times A)}$. The following is an immediate consequence of~\cite[Propositions~5.3 and~5.4, pp.~156--157]{BarrGrilletOsdol1971}.

\begin{lemma}\label{l:ker-order-emb}
Let $\cat{C}$ be a regular category and let $A$ be an object of $\cat{C}$. Then \[\ker\colon (\Quo{A})^\op\to \Equiv{A}\] is an order embedding between (large) posets.
\end{lemma}
The equivalence relations in the image of the map $\ker\colon (\Quo{A})^\op\to \Equiv{A}$ are called \emph{effective}. In general, there may be equivalence relations on $A$ that are not effective. This leads us to the following notion, first introduced in~\cite{BarrGrilletOsdol1971}.
\begin{definition}
A \emph{Barr-exact category} is a regular category in which every equivalence relation is effective.
\end{definition}

\begin{example}
Any variety of (possibly infinitary) algebras, with morphisms all the homomorphisms, is a Barr-exact category. The category of compact Hausdorff spaces and continuous maps is Barr-exact and, more generally, so is any category that is monadic over the category of sets. Any (elementary) topos is in particular a Barr-exact category.
\end{example}

\begin{remark}\label{rem:Barr-exact-ker-iso}
In view of~\cref{l:ker-order-emb}, a regular category is Barr-exact precisely when the map $\ker\colon (\Quo{A})^\op\to \Equiv{A}$ is an order isomorphism for all objects $A$.
\end{remark}

\begin{remark}
Suppose that $\cat{C}$ is a regular category admitting coequalisers of equivalence relations. Then, for every object $A$ of $\cat{C}$, there is a monotone map 
\[
\coeq \colon \Equiv{A} \to (\Quo{A})^\op
\]
sending an equivalence relation $\langle p_1,p_2\rangle \colon R\emb A\times A$ to the coequaliser of $p_1$ and~$p_2$. The map $\coeq$ is left inverse to $\ker\colon (\Quo{A})^\op\to \Equiv{A}$. Hence, $(\Quo{A})^\op$ is a retract of $\Equiv{A}$ in the category of (large) posets and monotone maps.
\end{remark}

\section{$\K$-sheaves}\label{K-sheaves}
As mentioned in the Introduction, $\ksat$-sheaves can be regarded as ``sheaves defined on compact saturated subsets''. To make this intuition more precise, consider a sheaf of sets on a topological space $X$. That is, a presheaf 
\[
F \colon \Omega(X)^\op \to \Set
\] 
satisfying the \emph{patch property}: for every set $\{U_i\mid i \in I\}\subseteq \Omega(X)$ of open subsets of $X$ and every tuple $(s_i)_{i \in I}\in \prod_{i \in I}{F(U_i)}$, if for all $i,j \in I$ we have $\ros{F}{U_i}{U_i \cap U_j}(s_i) = \ros{F}{U_j}{U_i \cap U_j}(s_j)$, then there exists a unique $s \in F(\bigcup_{i \in I} U_i)$ such that $\ros{F}{\bigcup_{i \in I} U_i}{U_j}(s) = s_j$ for all $j \in I$.

Sheaves on $X$ can be also described as \emph{\'etale spaces} over $X$, i.e.\ local homeomorphisms 
\[
p\colon E \to X.
\] 
In fact, any \'etale space induces a sheaf $\Omega(X)^\op \to \Set$ that sends $U\in \Omega(X)$ to the set of \emph{local sections} of $p$ over $U$, i.e.\ the continuous maps $s\colon U\to E$ such that $p\circ s = \mathrm{id}_U$. The functorial action is given by restricting local sections to open subsets of their domains. Conversely, a sheaf $F$ induces an \'etale space $p\colon E_F \to X$ where $E_F$ is obtained by ``gluing together'' the \emph{stalks} of $F$, and the map $p$ contracts the stalk of $F$ at a point $x\in X$ to~$x$. These assignments yield an equivalence between the category of sheaves of sets over $X$ and the category of \'etale spaces over $X$. See e.g.\ \cite[Corollary~II.6.3]{MacLaneMoerdijk1994}. 

If $\cat{V}$ is any variety of finitary algebras, we can consider the categories of internal $\cat{V}$-algebras in the toposes of $\Set$-valued sheaves over $X$ and of \'etale spaces over $X$, respectively. This yields an equivalence between \emph{sheaves of $\cat{V}$-algebras} over $X$ (that is, presheaves $\Omega(X)^\op \to \cat{V}$ satisfying the patch property) and \emph{\'etale spaces of $\cat{V}$-algebras} over $X$; cf.\ e.g.\ \cite[\S II.7]{MacLaneMoerdijk1994}.

Now, let $F\colon \Omega(X)^\op \to \Set$ be a sheaf with corresponding \'etale space $p\colon E \to X$. An advantage of the latter formulation is that we can consider sections of $p$ over arbitrary subsets of $X$, not just the open ones. Thus, we can associate to any compact saturated subset $K\subseteq X$ the set $G(K)$ of local sections of $p$ over $K$. This induces a presheaf $G$ over the poset of compact saturated subsets of $X$ ordered by inclusion. If $X$ is nice enough (e.g., it is locally compact), the sheaf $F$ can be recovered from $G$ via the isomorphism 
\[
F(U) \cong \lim_{K\subseteq U}{G(K)}
\]
where $K$ ranges over the compact saturated sets contained in $U$. See e.g.\ \cite[Lemma~3.4]{GehrkeGool2018}. The notion of $\ksat$-sheaf introduced in \cref{d:K-sheaf-on-complete-lattice} below captures precisely the properties of the presheaf $G$, as will become clear in \cref{s:K-vs-Omega}.

Throughout this section, we fix an arbitrary regular category $\cat{C}$, an object $A$ of $\cat{C}$, and a complete lattice $P$. 
The latter can be thought of as the lattice of closed subsets of a compact Hausdorff space or, more generally, the lattice of compact saturated subsets of a \emph{stably compact} space (see \cref{ex:stably-compact} for a definition).

The following notion, along with the functor $\gamma_*$ defined in \cref{eq:gamma-direct-image} below, will play a central role in the remainder of the paper.

\begin{definition}
The \emph{canonical representation} of $A$ is the functor 
	\[
	\gamma\colon \RegEpi{A} \to \cat{C}
	\]
	sending a regular epimorphism $A\repi B$ to its codomain (i.e., $\gamma$ is the restriction of the codomain functor $A/{\cat{C}}\to \cat{C}$).
\end{definition}

\begin{remark}
The name ``canonical representation'' stems from the observation that, under certain conditions, the poset reflection of $\RegEpi{A}$ can be thought of as the opposite of the poset of compact saturated subsets of a space $X$, and $\gamma(k)$ as the set of local sections over $k$ of a sheaf over $X$ whose object of global sections is isomorphic to $A$.
\end{remark}

The functor category $[P^\op,\RegEpi{A}]$ can be identified with the large preorder of monotone maps $P^\op\to \RegEpi{A}$, with respect to the pointwise preorder. The canonical representation of $A$ induces a ``direct image'' functor $\gamma_*$ given by post-composing with $\gamma$:
\begin{equation}\label{eq:gamma-direct-image}
\begin{aligned}
\gamma_*\colon [P^\op,\RegEpi{A}] &\to [P^\op, \cat{C}] \\ 
H &\mapsto \gamma\circ H
\end{aligned}
\ \ \ \ \ \ \ \ \ \
\begin{tikzcd}
\RegEpi{A} \arrow{r}{\gamma} & \cat{C} \\
P^\op \arrow{u}{H} \arrow{ur}[swap]{\gamma_* H} & 
\end{tikzcd}
\end{equation}
A morphism $\alpha\colon H \Rightarrow J$ in $[P^\op,\RegEpi{A}]$ is sent by $\gamma_*$ to the horizontal composition of natural transformations $\mathrm{id}_{\gamma}\alpha \colon \gamma_* H \Rightarrow \gamma_* J$.

In this section we shall see that order-theoretic properties of the monotone map $H\colon P^\op\to \RegEpi{A}$ correspond to certain sheaf-like properties---made precise in the following definition---of the associated presheaf $\gamma_* H\colon P^\op\to\cat{C}$. 
\begin{definition} \label{d:K-sheaf-on-complete-lattice}
	A \emph{$\cat{C}$-valued $\K$-sheaf} over $P$ is a functor $F \colon P^\op \to \cat{C}$ satisfying the following properties:
	\begin{enumerate}[label=\textnormal{(K\arabic*)}]
	\item \label{K-empty} $F(\bot)$ is a subterminal object of $\cat{C}$, i.e.\ the unique arrow $F(\bot)\to\1$ is monic.
	\item \label{K-pullback} For all $p, q \in P$, the following is a pullback square in $\cat{C}$:
	\[\begin{tikzcd}[column sep = 4em]
	F(p \vee q) \arrow{r}{\ros{F}{p \vee q}{ p}} \arrow[swap]{d}{\ros{F}{p \vee q}{ q}} \arrow[dr, phantom, "\lrcorner", very near start, xshift=-6pt] & F(p) \arrow{d}{\ros{F}{p}{ p \wedge q}}\\
	F(q) \arrow{r}{\ros{F}{q}{ p \wedge q}} & F(p \wedge q)
	\end{tikzcd}\]
	\item\label{K-codirected-meet} $F$ preserves directed colimits. I.e., for all codirected subsets $\D\subseteq P$, the~cocone $(\ros{F}{p}{\inf \D}\colon F(p) \to F(\inf \D))_{p \in \D}$ is a colimit of the restriction of $F$ to~$\D$.
\end{enumerate}
We let $\KSh(P,\cat{C})$ denote the full subcategory of $[P^\op, \cat{C}]$ defined by the $\K$-sheaves. 
\end{definition}

\begin{remark}
The definition of $\K$-sheaf given above is a slight variant of the homonymous notion introduced by Lurie~\cite[Definition~7.3.4.1]{Lurie2009}. In \emph{op.\ cit.}, $\cat{C}$ is an \mbox{\emph{$\infty$-category}} and $P$ is the lattice of compact subsets of a locally compact Hausdorff space.
Note that Lurie requires $F(\bot)$ to be a terminal object of $\cat{C}$, whereas we relax this condition by replacing ``terminal'' with ``subterminal''. We do this so that \cref{cor:MAIN} below (which generalises the main result of \cite{GehrkeGool2018} for varieties of finitary algebras) holds for all objects of $\cat{C}$, including those objects $A$ such that the unique morphism $A\to\1$ is not a regular epimorphism. E.g., if $\cat{C}$ is the variety of semigroups and $A$ is the empty semigroup, i.e.\ the initial object of $\cat{C}$, the unique morphism $A\to\1$ is not a regular epimorphism.
\end{remark}

We shall also consider the following variant of condition~\ref{K-pullback}:
\begin{enumerate}[label=\textnormal{(K4)}]
\item \label{K-pushout} For all $p, q \in P$, the following is a pushout square in $\cat{C}$:
\[\begin{tikzcd}[column sep = 4em]
F(p \vee q) \arrow{r}{\ros{F}{p \vee q}{ p}} \arrow[swap]{d}{\ros{F}{p \vee q}{ q}} & F(p) \arrow{d}{\ros{F}{p}{ p \wedge q}}\\
F(q) \arrow{r}{\ros{F}{q}{ p \wedge q}} & F(p \wedge q) \arrow[ul, phantom, "\ulcorner", very near start]
\end{tikzcd}\]
\end{enumerate}

\begin{proposition}\label{p:properties-H-presheaf}
Let $H\colon P^\op\to\RegEpi{A}$ be a monotone map. 
	The following statements hold: 
	\begin{enumerate}[label=(\alph*)]
	\item\label{empty-inf} $H$ preserves the infimum of the empty set\footnote{That is, $H$ sends the greatest element of $P^\op$ (equivalently, the least element $\bot$ of $P$) to a maximum in $\RegEpi{A}$. Similarly for the following items.} if and only if $\gamma_* H$ satisfies~\ref{K-empty}.
	\item\label{binary-infima} $H$ preserves binary infima if and only if, for all $p,q\in P$, the following mediating morphism is monic: 
	\[
	\gamma_* H(p\vee q)\to \gamma_* H (p)\times \gamma_* H (q).
	\] 
	\item\label{binary-sup} $H$ preserves binary suprema if and only if $\gamma_* H$ satisfies~\ref{K-pushout}.
	\item\label{directed-sup} $H$ preserves directed suprema if and only if $\gamma_* H$ satisfies~\ref{K-codirected-meet}.
	\end{enumerate}
\end{proposition}
\begin{proof}
Consider an arbitrary finite set $\{p_i\mid i\in I\}\subseteq P$. We claim that $H(\sup_{i\in I}{p_i})$ is an infimum of $\{H(p_i)\mid i\in I\}$ if, and only if, the induced mediating morphism 
\[
\gamma_* H\bigg(\sup_{i\in I}{p_i}\bigg)\to \prod_{i\in I}{\gamma_* H (p_i)}
\] 
is monic. \Cref{empty-inf,binary-infima} then follow at once by letting $\{p_i\mid i\in I\}$ be the empty set and any two-element set, respectively.
Recall that an infimum $t\colon A\repi B$ of $\{H(p_i)\mid i\in I\}$ is obtained by taking the (regular epi, mono) factorisation of the arrow $A\to \prod_{i\in I}{\gamma_* H (p_i)}$ whose composition with the $i$th projection is $H(p_i)$. Thus, $H(\sup_{i\in I}{p_i})$ is an infimum of $\{H(p_i)\mid i\in I\}$ precisely when it is above $t$. Consider the following commutative square.
\[\begin{tikzcd}[row sep = 2em]
A \arrow[twoheadrightarrow]{d}[swap]{H(\sup_{i\in I}{p_i})} \arrow[twoheadrightarrow]{r}{t} & B \arrow[rightarrowtail]{d} \arrow[dashed]{dl} \\
\gamma_* H(\sup_{i\in I}{p_i}) \arrow{r} & \prod_{i\in I}{\gamma_* H (p_i)}
\end{tikzcd}\]
If the bottom horizontal arrow is monic, there is a diagonal filler $B\to \gamma_* H(\sup_{i\in I}{p_i})$, showing that $H(\sup_{i\in I}{p_i})$ is above $t$. Conversely, if $H(\sup_{i\in I}{p_i})$ is above $t$ there is $h\colon B\to \gamma_* H(\sup_{i\in I}{p_i})$ making the upper triangle commute, and such an  $h$ is necessarily a regular epimorphism. Using the fact that the outer square commutes and $t$ is an epimorphism, we see that the lower triangle must also commute. It follows that $h$ is also monic, hence an isomorphism, and so the bottom horizontal arrow is monic.

For \cref{binary-sup}, note that $H$ preserves binary suprema if and only if, for all $p,q\in P$, $H(p\wedge q)$ is a supremum of $H(p)$ and $H(q)$. In turn, by \cref{l:binary-infima}, this is equivalent to saying that the left-hand diagram below is a pushout in $\cat{C}$.
\begin{center}
\begin{tikzcd}[column sep = 4em]
A \arrow[twoheadrightarrow]{r}{H(p)} \arrow[twoheadrightarrow]{d}[swap]{H(q)} & \gamma_* H(p) \arrow{d}{\ros{\gamma_* H}{p}{p \wedge q}}\\
\gamma_* H(q) \arrow{r}{\ros{\gamma_* H}{q}{p \wedge q}} & \gamma_* H(p \wedge q) \arrow[ul, phantom, "\ulcorner", very near start]
\end{tikzcd}
\ \ \ \ \ \ 
\begin{tikzcd}[column sep = 4em]
A \arrow[twoheadrightarrow]{dr}[description]{H(p\vee q)} \arrow[bend left = 30, looseness=1, two heads]{rrd}[description]{H(p)} \arrow[bend right = 30, looseness=1, two heads]{rdd}[description]{H(q)} & & \\
& \gamma_* H(p\vee q) \arrow{d}[swap]{\ros{\gamma_* H}{p \vee q}{ q}} \arrow{r}{\ros{\gamma_* H}{p \vee q}{ p}} & \gamma_* H(p) \\
& \gamma_* H(q) &
\end{tikzcd}
\end{center}
Both $H(p)$ and $H(q)$ factor through the regular epimorphism 
\[
H(p\vee q)\colon A\repi \gamma_*H(p\vee q),
\]
as depicted in the rightmost diagram above.
Hence the leftmost diagram above is a pushout precisely when $\gamma_* H$ satisfies~\ref{K-pushout}.

For \cref{directed-sup}, let $\D$ be a codirected subset of $P$. We must show that $H(\inf \D)$ is a supremum of $\{H(p)\mid p\in \D\}$ if, and only if, the cocone 
\begin{equation}\label{eq:candidate-colimit-cocone}
(\ros{\gamma_* H}{p}{\inf \D}\colon \gamma_* H(p) \to \gamma_* H(\inf \D))_{p \in \D}
\end{equation}
is a colimit of the restriction of $\gamma_* H$ to $\D$. 
Suppose the latter is a colimit cocone.
Clearly, $H(\inf \D)$ is above $H(p)$ for all $p\in \D$, so assume that $f\colon A\repi B$ is an element of $\RegEpi{A}$ that is above $H(p)$ for all $p\in \D$. That is, for each $p\in\D$ there is $g_p \colon \gamma_* H(p)\to B$ such that $g_p\circ H(p) = f$. Then $(g_p \colon \gamma_* H(p) \to B)_{p \in \D}$ is a compatible cocone over the diagram given by the restriction of $\gamma_* H$ to $\D$. Just observe that, for all $p,q \in \D$ such that $p\leq q$, 
\[
g_p \circ \ros{\gamma_* H}{q}{p} \circ H(q) = g_p \circ H(p) = f = g_q \circ H(q)
\]
and so $g_p \circ \ros{\gamma_* H}{q}{p} = g_q$ because $H(q)$ is an epimorphism.
Hence there is a unique arrow $j\colon \gamma_* H(\inf \D) \to B$ satisfying $j\circ \ros{\gamma_* H}{p}{\inf \D} = g_p$ for all $p\in\D$. We then see that $f$ is above $H(\inf \D)$ because, if $p$ is an arbitrary element of $\D$,
\[
j\circ H(\inf \D) = j\circ \ros{\gamma_* H}{p}{\inf\D} \circ H(p) = g_p \circ H(p) = f.
\]
Therefore, $H(\inf \D)$ is a supremum of $\{H(p)\mid p\in \D\}$.

Conversely, suppose that $H(\inf \D)$ is a supremum of $\{H(p)\mid p\in \D\}$ and consider a compatible cocone $(h_p \colon \gamma_* H(p) \to B)_{p \in \D}$ over the diagram given by the restriction of $\gamma_* H$ to $\D$. Fix an arbitrary $p\in\D$ and take the (regular epi, mono) factorisation of $h_p\circ H(p)$, as depicted in the following diagram.
\begin{equation}\label{eq:factor-H(p)hp}
\begin{tikzcd}
A \arrow[twoheadrightarrow]{r}{H(p)} \arrow[twoheadrightarrow]{d}[swap]{f} & \gamma_* H(p) \arrow{d}{h_p} \\
C \arrow[rightarrowtail]{r}{m} & B
\end{tikzcd}
\end{equation}
This yields $f\in \RegEpi{A}$. Note that, because $\D$ is codirected, $f$ does not depend on the choice of $p$. To see this, pick another $q\in \D$. As $\D$ is codirected, there is $r\in \D$ that is below $p$ and $q$. So,
\[
h_r \circ H(r) = h_r \circ \ros{\gamma_* H}{p}{r} \circ H(p) = h_p \circ H(p)
\]
and, by a similar reasoning, $h_r \circ H(r) = h_q \circ H(q)$. Thus, $h_p \circ H(p) = h_q \circ H(q)$. Now, observe that the square in \cref{eq:factor-H(p)hp} admits a diagonal filler $\gamma_* H(p)\to C$. Because $p\in \D$ was chosen arbitrarily, it follows that $f$ is above $H(p)$ for all $p\in \D$. Since $H(\inf \D)$ is a supremum of $\{H(p)\mid p\in \D\}$, there is $j\colon \gamma_* H(\inf \D)\to C$ such that $j\circ H(\inf \D) = f$. The composite $m\circ j\colon \gamma_* H(\inf \D)\to B$ satisfies
\[
m\circ j \circ \ros{\gamma_* H}{p}{\inf\D} \circ H(p) = m\circ j \circ H(\inf{\D}) = m\circ f = h_p \circ H(p) 
\]
and so $m\circ j \circ \ros{\gamma_* H}{p}{\inf\D} = h(p)$ because $H(p)$ is an epimorphism. In other words, $m\circ j$ is a morphism from the cocone in \cref{eq:candidate-colimit-cocone} to the cocone $(h_p \colon \gamma_* H(p) \to B)_{p \in \D}$. Finally, observe that $m\circ j$ is the unique such morphism, for if $n\colon \gamma_* H(\inf \D)\to B$ is another morphism of cocones then, for all $p\in \D$, 
\[
n\circ \ros{\gamma_* H}{p}{\inf{\D}} = h_p = m\circ j \circ \ros{\gamma_* H}{p}{\inf{\D}}
\]
entails $n = m\circ j$ because $\ros{\gamma_* H}{p}{\inf{\D}}$ is an epimorphism (just note that $\ros{\gamma_* H}{p}{\inf{\D}}\circ H(p) = H(\inf\D)$). Therefore, the cocone in \cref{eq:candidate-colimit-cocone} is a colimit of the restriction of $\gamma_* H$ to $\D$.
\end{proof}

Given relations 
\[
\langle r_1,r_2\rangle \colon R\emb A\times A \ \ \text{ and } \ \ \langle s_1,s_2\rangle \colon S\emb A\times A
\] 
on $A$, we can define their composition $R\circ S$ as follows. Consider the following equaliser diagram:
\[\begin{tikzcd}[column sep=3.5em]
U \arrow[rightarrowtail]{r}{u} & R\times S \arrow[yshift=3pt]{r}{r_2\circ \pi_R} \arrow[yshift=-3pt]{r}[swap]{s_1\circ \pi_S} & A.
\end{tikzcd}\]
Then the (regular epi, mono) factorisation of the morphism \[\langle r_1\circ \pi_R\circ u, s_2\circ \pi_S\circ u\rangle\colon U\to A\times A\]
yields the composite relation $U\repi R\circ S\emb A\times A$. See e.g.\ \cite[\S1.56]{FreydScedrov1990}.

A classical result of universal algebra states that, for any two congruences $\theta_1, \theta_2$ on an algebra, the composite $\theta_1\circ \theta_2$ is a congruence if, and only if, $\theta_1$ and $\theta_2$ commute (i.e.\ $\theta_1\circ \theta_2=\theta_2\circ \theta_1$). 
In our setting, congruences correspond to effective equivalence relations. 
If the composition of two effective equivalence relations $\ker{f}$ and $\ker{g}$ is an effective equivalence relation, $\ker{f}$ and $\ker{g}$ commute (i.e.\ $\ker{f}\circ \ker{g}=\ker{g}\circ \ker{f}$). Conversely, if $\ker{f}$ and $\ker{g}$ commute, their composition $\ker{f}\circ \ker{g}$ is an equivalence relation; however, $\ker{f}\circ \ker{g}$ need not be effective. This leads us to the following notion:

\begin{definition}
Let $f,g\in\RegEpi{A}$. We say that $f$ and $g$ \emph{ker-commute} if the relations $\ker{f}$ and $\ker{g}$ commute and their composition is an effective equivalence relation.
\end{definition}

\begin{remark}
If $\cat{C}$ is Barr-exact, two regular epimorphisms $f,g\in\RegEpi{A}$ ker-commute precisely when the relations $\ker{f}$ and $\ker{g}$ commute. 
In fact, Barr-exact categories coincide with the regular categories in which the composition of any pair of commuting effective equivalence relations is an effective equivalence relation \cite[Theorem~17]{BurgessCaicedo1981}.
\end{remark}

To state the next lemma, we recall the following terminology from~\cite{Bourn2003} (cf.\ also~\cite{CKP1993}). In any regular category, a commutative diagram of regular epimorphisms
\[\begin{tikzcd}
B \arrow[twoheadrightarrow]{r}{f} \arrow[twoheadrightarrow]{d}[swap]{g} & C \arrow[twoheadrightarrow]{d}{h} \\
D \arrow[twoheadrightarrow]{r}{i} & E
\end{tikzcd}\]
is a \emph{regular pushout} if the unique mediating morphism from $B$ to the pullback of $h$ along $i$ is a regular epimorphism. (This nomenclature is justified by virtue of the fact that any regular pushout in a regular category is a pushout \cite[p.~118]{Bourn2003}.)
For a proof of the next result, cf.~\cite[Lemma~9 and Propositions~10 and~12]{BurgessCaicedo1981}.

\begin{lemma}\label{l:commuting}
The following statements are equivalent for all $f, g\in \RegEpi{A}$:
\begin{enumerate}
\item $f$ and $g$ ker-commute.
\item $f$ and $g$ admit a supremum $h$ in $\RegEpi{A}$ and $\ker{f}\circ \ker{g}=\ker{h}$.
\item The pushout in $\cat{C}$ of $f$ along $g$ exists and is a regular pushout.
\end{enumerate}
\end{lemma}

\begin{remark}
With regards to the previous result, a related fact was proved by Fay in~\cite{Fay1977}, motivated by an earlier unpublished version of~\cite{BurgessCaicedo1981} whose main results were announced in~\cite{BCnotices1972}.
\end{remark}

\begin{lemma}\label{l:monic-prod-implies-pullback}
	Let $H\colon P^\op\to \RegEpi{A}$ be a monotone map. 
	The following statements are equivalent: 
	\begin{enumerate}
	\item\label{i:H-bin-inf-sup-ker-comm} $H$ preserves binary infima and binary suprema, and its image consists of pairwise ker-commuting elements.
	\item\label{i:gamma-star-H-sat-K2} $\gamma_* H \colon P^\op \to \cat{C}$ satisfies~\ref{K-pullback}.
	\end{enumerate}
\end{lemma}
\begin{proof}
Assume that \ref{i:H-bin-inf-sup-ker-comm} holds and consider the following commutative diagram, where $\chi$ is the unique mediating morphism induced by the universal property of the pullback (recall that every arrow in the image of $\gamma_* H$ is a regular epimorphism).
\begin{equation}\label{eq:pullback-pushout}
\begin{tikzcd}[column sep= 4.2em]
\gamma_* H(p \vee q) \arrow[bend left = 20, looseness=1, two heads]{rrd}[description]{\ros{\gamma_* H}{p \vee q}{ p}} \arrow[swap, bend right = 30, looseness=1, two heads]{ddr}[description]{\ros{\gamma_* H}{p \vee q}{ q}}\arrow[dashed]{rd}[description]{\chi} & & \\
& P \arrow[twoheadrightarrow]{r}{p_1} \arrow[twoheadrightarrow]{d}[swap]{p_2} \arrow[dr, phantom, "\lrcorner", very near start] & \gamma_* H(p) \arrow[twoheadrightarrow]{d}{\ros{\gamma_* H}{p}{ p \wedge q}} \\
& \gamma_* H(q) \arrow[twoheadrightarrow]{r}{\ros{\gamma_* H}{q}{ p \wedge q}} & \gamma_* H(p \wedge q)
\end{tikzcd}
\end{equation}
 The composite
\[\begin{tikzcd}[column sep = 3em]
\gamma_* H(p \vee q) \arrow{r}{\chi} & P \arrow{r}{\langle p_1,p_2\rangle} & \gamma_* H (p)\times \gamma_* H (q)
\end{tikzcd}\]
coincides with the pairing of $\ros{\gamma_* H}{p \vee q}{ p}$ and $\ros{\gamma_* H}{p \vee q}{ q}$. As $H$ preserves binary infima, the latter pairing is a monomorphism by \cref{binary-infima} in \cref{p:properties-H-presheaf}, and so $\chi$ is monic. 
On the other hand, because $H$ preserves binary suprema, the outer diagram in \cref{eq:pullback-pushout} is a pushout by~\cref{binary-sup} in \cref{p:properties-H-presheaf}. Thus, the diagram obtained by precomposing with the (regular) epimorphism $H(p\vee q)\colon A \repi \gamma_* H(p \vee q)$ is also a pushout:
\[\begin{tikzcd}[column sep = 4.2em]
A \arrow[twoheadrightarrow]{r}{H(p)} \arrow[twoheadrightarrow]{d}[swap]{H(q)} & \gamma_* H(p) \arrow[twoheadrightarrow]{d}{\ros{\gamma_* H}{p}{p\wedge q}} \\
\gamma_* H(q) \arrow[twoheadrightarrow]{r}{\ros{\gamma_* H}{q}{p\wedge q}} & \gamma_* H(p\wedge q) \arrow[ul, phantom, "\ulcorner", very near start]
\end{tikzcd}\]
The unique arrow $A \to P$ induced by the universal property of the pullback coincides with $\chi\circ H(p\vee q)$ and is a regular epimorphism by \cref{l:commuting}, since $H(p)$ and $H(q)$ ker-commute. It follows from \cref{cancellation} in \cref{regular-properties} that $\chi$ is a regular epimorphism, hence an isomorphism. Thus the outer diagram in \cref{eq:pullback-pushout} is a pullback, i.e.\ $\gamma_* H$ satisfies \ref{K-pullback}.

Conversely, suppose that \ref{i:gamma-star-H-sat-K2} holds and fix arbitrary elements $p,q \in P$.
As $\gamma_* H$ satisfies \ref{K-pullback}, the following is a pullback square.
\begin{equation}\label{eq:pullback-sq-gammastar}
\begin{tikzcd}[column sep = 4em]
\gamma_* H(p \vee q) \arrow[twoheadrightarrow]{r}{\ros{\gamma_* H}{p \vee q}{ p}} \arrow[twoheadrightarrow,swap]{d}{\ros{\gamma_* H}{p \vee q}{ q}} \arrow[xshift=-6pt, dr, phantom, "\lrcorner", very near start] & \gamma_* H(p) \arrow[twoheadrightarrow]{d}{\ros{\gamma_* H}{p}{ p \wedge q}}\\
\gamma_* H(q) \arrow[twoheadrightarrow]{r}{\ros{\gamma_* H}{q}{ p \wedge q}} & \gamma_* H(p \wedge q)
\end{tikzcd}
\end{equation}
The pairing of $\ros{\gamma_*H}{p \vee q}{ p}$ and $\ros{\gamma_*H}{p \vee q}{ q}$ is then an equaliser of the arrows
\[\begin{tikzcd}[column sep = 4em]
\gamma_* H (p)\times \gamma_* H (q) \arrow{r} & \gamma_* H (p) \arrow{r}{\ros{\gamma_* H}{p}{ p \wedge q}} & \gamma_* H(p \wedge q)
\end{tikzcd}\]
and 
\[\begin{tikzcd}[column sep = 4em]
\gamma_* H (p)\times \gamma_* H (q) \arrow{r} & \gamma_* H (q) \arrow{r}{\ros{\gamma_* H}{q}{ p \wedge q}} & \gamma_* H(p \wedge q),
\end{tikzcd}\]
where the first morphisms in the two compositions are the appropriate product projections. 
In particular, the mediating morphism 
\[
\gamma_* H(p\vee q)\to \gamma_* H (p)\times \gamma_* H (q)
\] 
given by the pairing of $\ros{\gamma_*H}{p \vee q}{ p}$ and $\ros{\gamma_*H}{p \vee q}{ q}$ is monic. It follows from \cref{binary-infima} in \cref{p:properties-H-presheaf} that $H$ preserves binary infima.
Moreover, the square in \cref{eq:pullback-sq-gammastar} is a pushout by \cref{pullback-pushout} in \cref{regular-properties}, and so an application of \cref{binary-sup} in \cref{p:properties-H-presheaf} shows that $H$ preserves binary suprema.
Finally, in order to prove that $H(p)$ and $H(q)$ ker-commute, consider the following commutative diagram.
\begin{equation*}
	\begin{tikzcd}[column sep= 3.8em]
		A \arrow[bend left = 30, looseness=1, two heads]{rrd}[description]{H(p)} \arrow[swap, bend right = 30, looseness=1, two heads]{ddr}[description]{H(q)}\arrow[twoheadrightarrow]{rd}[description]{H(p \lor q)} & & \\
		& \gamma_* H(p \vee q) \arrow[twoheadrightarrow]{r}{\ros{\gamma_*H}{p \vee q}{ p}} \arrow[twoheadrightarrow,swap]{d}{\ros{\gamma_*H}{p \vee q}{ q}} \arrow[dr, phantom, "\lrcorner", very near start] & \gamma_* H(p) \arrow[twoheadrightarrow]{d}{\ros{\gamma_* H}{p}{ p \wedge q}} \\
		& \gamma_* H(q) \arrow[twoheadrightarrow]{r}{\ros{\gamma_* H}{q}{ p \wedge q}} & \gamma_* H(p \wedge q)
	\end{tikzcd}
\end{equation*}
The inner square is a pushout, and the morphism $H(p\vee q)\colon A \repi \gamma_* H(p \vee q)$ is an epimorphism (in fact, a regular epimorphism).
Therefore, the outer square is a pushout.
Since the inner square is a pullback and the morphism $H(p\vee q)\colon A \repi \gamma_* H(p \vee q)$ is a regular epimorphism, we see that the outer square is a regular pushout.
Therefore, $H(p)$ and $H(q)$ ker-commute by \cref{l:commuting}.
\end{proof}

We thus obtain a characterisation of those monotone maps $H\colon P^\op\to \RegEpi{A}$ such that the presheaf $\gamma_* H \colon P^\op\to\cat{C}$ is a $\K$-sheaf.
\begin{theorem}\label{thm:gamma-star-H-sheaf}
Consider a monotone map $H\colon P^\op\to \RegEpi{A}$.
The following statements are equivalent:
\begin{enumerate}
\item \label{i:H}
$H$ preserves finite infima and non-empty suprema, and its image consists of pairwise ker-commuting elements.
\item \label{i:gamma-H}
$\gamma_* H \colon P^\op\to\cat{C}$ is a $\K$-sheaf. 
\end{enumerate}
\end{theorem}
\begin{proof}	
In view of \cref{empty-inf,directed-sup} in \cref{p:properties-H-presheaf}, combined with \cref{l:monic-prod-implies-pullback}, $\gamma_* H$ satisfies \ref{K-empty}--\ref{K-codirected-meet} if, and only if, $H$ preserves finite infima and non-empty suprema and its image consists of pairwise ker-commuting elements. Just observe that $H$ preserves non-empty suprema precisely when it preserves binary and directed suprema.
\end{proof}

\begin{corollary}\label{cor:gamma-star-H-sheaf-repr}
Consider a monotone map $H\colon P^\op\to \RegEpi{A}$.
The following statements are equivalent:	
\begin{enumerate}
		\item
		$H$ preserves finite infima and arbitrary suprema, and its image consists of pairwise ker-commuting elements.
		\item
		$\gamma_* H \colon P^\op\to\cat{C}$ is a $\K$-sheaf and $H(\top)$ is an isomorphism.
	\end{enumerate}
\end{corollary}
\begin{proof}
This is an immediate consequence of \cref{thm:gamma-star-H-sheaf}. Just observe that $H$ preserves the supremum of the empty set if, and only if,
\[
H(\top)\colon A \to \gamma_* H(\top)
\]
is a minimum in $\RegEpi{A}$, i.e.\ an isomorphism in $\cat{C}$.
\end{proof}

\section{Soft $\K$-sheaf Representations}\label{s:soft-K-sheaf-repr}
As in the previous section, we fix an arbitrary regular category $\cat{C}$, an object $A$ of $\cat{C}$, and a complete lattice $P$.

To start with, we introduce the concept of softness for presheaves. We note that, by~\cref{cancellation} in~\cref{regular-properties}, all arrows in the image of the canonical representation ${\gamma\colon \RegEpi{A}\to \cat{C}}$ are regular epimorphisms. 
Therefore, this property is inherited by all presheaves of the form $\gamma_* H$ with $H\in [P^\op,\RegEpi{A}]$. The following observation is straightforward:
\begin{lemma}\label{l:soft-weakening}
	The following statements are equivalent for any presheaf $F\in [P^\op,\cat{C}]$:
	\begin{enumerate}
	\item\label{soft_top} $\ros{F}{\top}{ p}\colon F(\top)\to F(p)$ is a regular epimorphism for all $p\in P$.
	\item\label{soft-pairwise} $\ros{F}{q}{ p}\colon F(q)\to F(p)$ is a regular epimorphism for all $p,q\in P$ with $p\leq q$.
	\end{enumerate}
\end{lemma}
\begin{proof}
Clearly, \ref{soft-pairwise} implies \ref{soft_top}. Conversely, suppose \ref{soft_top} holds and let $p,q\in P$ satisfy ${p\leq q}$.
By functoriality of $F$ we have $\ros{F}{q}{ p} \circ \ros{F}{\top}{ q} = \ros{F}{\top}{p}$, which is a regular epimorphism. It follows from \cref{cancellation} in \cref{regular-properties} that $\ros{F}{q}{p}$ is a regular epimorphism.
\end{proof}

\begin{definition}
	A presheaf $P^\op\to\cat{C}$ is said to be \emph{soft} if it satisfies either of the equivalent conditions in~\cref{l:soft-weakening}.
\end{definition}

\begin{remark}\label{rem:soft-implies-K4}
By \cref{pullback-pushout} in \cref{regular-properties}, any soft $\K$-sheaf satisfies \ref{K-pushout}.
\end{remark}

Next, we look at those $\K$-sheaves $P^\op\to\cat{C}$ which allow us to recover $A$, up to isomorphism, as the object of global sections.

\begin{definition}
	A \emph{$\K$-sheaf representation} of $A$ over $P$ is a pair $(F,\phi)$ where $F\colon P^\op\to \cat{C}$ is a $\K$-sheaf and $\phi\colon A \to F(\top)$ is an isomorphism in $\cat{C}$. If $F$ is soft, we call $(F,\phi)$ a \emph{soft $\K$-sheaf representation} of $A$ over $P$.
	
	Denote by 
	\[
	\KShr{A}(P,\cat{C})
	\] 
	the category of $\K$-sheaf representations of $A$ over $P$. The objects of $\KShr{A}(P,\cat{C})$ are $\K$-sheaf representations $(F,\phi)$ of $A$ over $P$, and a morphism $(F,\phi)\to (G,\psi)$ is a natural transformation $\alpha\colon F\Rightarrow G$ satisfying $\alpha_{\top}\circ \phi=\psi$. The full subcategory of $\KShr{A}(P,\cat{C})$ defined by the \emph{soft} $\K$-sheaf representations is denoted by 
	\[
	\sKShr{A}(P,\cat{C}).
	\]
\end{definition}

\begin{lemma}\label{l:soft-repr-thin}
The following statements hold:
\begin{enumerate}[label=(\alph*)]
\item\label{soft-domain} For any two objects $(F,\phi), (G,\psi)\in \KShr{A}(P,\cat{C})$, if $(F,\phi)$ is a soft $\K$-sheaf representation, there is at most one morphism $(F,\phi)\to (G,\psi)$ in $\KShr{A}(P,\cat{C})$.
\item\label{soft-thin} $\sKShr{A}(P,\cat{C})$ is a (large) preorder.
\end{enumerate}
\end{lemma}
\begin{proof}
For~\cref{soft-domain}, consider any two arrows $\alpha,\beta\colon (F,\phi)\rightrightarrows (G,\psi)$ in $\KShr{A}(P,\cat{C})$ such that $F$ is soft. We claim that $\alpha_p=\beta_p$ for all $p\in P$, and so $\alpha=\beta$.
For every $p\in P$, we have
\begin{align*}
\alpha_p\circ \ros{F}{\top}{p} \circ \phi &= {\ros{G}{\top}{p}} \circ \alpha_{\top} \circ \phi \tag*{Naturality of $\alpha$} \\
&= {\ros{G}{\top}{p}} \circ \psi \\
&= {\ros{G}{\top}{p}} \circ \beta_{\top} \circ \phi \\
&= \beta_p\circ {\ros{F}{\top}{p}} \circ \phi. \tag*{Naturality of $\beta$}
\end{align*}
As $\phi$ is an isomorphism we get $\alpha_p\circ {\ros{F}{\top}{p}} = \beta_p\circ {\ros{F}{\top}{p}}$. Since $F$ is soft, $\ros{F}{\top}{p}$ is an epimorphism. Hence, $\alpha_p = \beta_p$.

\Cref{soft-thin} is an immediate consequence of~\cref{soft-domain}.
\end{proof}

\begin{proposition}\label{p:sheaf-to-monot-map}
	Let $(F,\phi)\in \sKShr{A}(P,\cat{C})$ and define the monotone map
	\[
	H_F \colon P^\op \to \RegEpi{A}, \ \ p\mapsto (A \xrightarrow{\phi} F(\top) \xrightarrow{\ros{F}{\top}{ p}} F(p)).
	\]
	The following statements hold:
	\begin{enumerate}[label=(\alph*)]
	\item\label{gammastar-surj} $\gamma_* H_F = F$.
	\item\label{HF-preservation} $H_F$ preserves finite infima and arbitrary suprema, and its image consists of pairwise ker-commuting elements.
	\end{enumerate}
\end{proposition}
\begin{proof}
For \cref{gammastar-surj}, note that an arrow $p\to q$ in $P$ is sent by $\gamma_* H_F$ to the unique arrow $h\colon F(q)\to F(p)$ in $\cat{C}$ such that the following diagram commutes.
\[\begin{tikzcd}
& A \arrow[twoheadrightarrow]{dl}[swap]{H_F(q)} \arrow[twoheadrightarrow]{dr}{H_F(p)} & \\
F(q) \arrow{rr}{h} & & F(p)
\end{tikzcd}\]
It follows that $h = {\ros{F}{q}{p}}$ because
\[
{\ros{F}{q}{p}} \circ H_F(q) = {\ros{F}{q}{p}} \circ {\ros{F}{\top}{p}} \circ \phi =  {\ros{F}{\top}{p}} \circ \phi  = H_F(p),
\]
and so $\gamma_* H_F = F$.

\Cref{HF-preservation} is an immediate consequence of \cref{gammastar-surj} combined with \cref{cor:gamma-star-H-sheaf-repr}. Just observe that $H_F(\top)=\phi$ is an isomorphism.
\end{proof}

The following is our main result concerning soft $\K$-sheaf representations of objects of regular categories. Recall that $P$ denotes an arbitrary complete lattice, and $A$ an object of a regular category~$\cat{C}$.
\begin{theorem}\label{th:equiv-soft-sheaf-repr}
Let $\cat{M}$ be the (large) sub-preorder of $[P^\op,\RegEpi{A}]$ consisting of those maps that preserve finite infima and arbitrary suprema, and whose images consist of pairwise ker-commuting elements. The functor $\gamma_* \colon [P^\op,\RegEpi{A}] \to [P^\op, \cat{C}]$ in~\cref{eq:gamma-direct-image} induces an isomorphism
\[
\cat{M} \cong \sKShr{A}(P,\cat{C}).
\]
\end{theorem}
\begin{proof}
Define the functor 
\[
\Xi\colon \cat{M} \to \sKShr{A}(P,\cat{C}), \ \ H \mapsto (\gamma_* H, H(\top)).
\]
An arrow $\alpha\colon H\Rightarrow J$ in $\cat{M}$ is sent to $\gamma_* \alpha$. 
By \cref{cor:gamma-star-H-sheaf-repr}, the presheaf $\gamma_* H$ is a (soft) $\K$-sheaf and $H(\top)$ is an isomorphism in $\cat{C}$.
Thus, $(\gamma_* H, H(\top))$ is a soft $\K$-sheaf representation of $A$. If $\alpha\colon H\Rightarrow J$ is an arrow in $\cat{M}$ then $\alpha_{\top}\circ H(\top) = J(\top)$ in $\cat{C}$, and so $(\gamma_*\alpha)_{\top} \circ H(\top) = J(\top)$. Therefore, the functor $\Xi$ is well-defined.

In view of \cref{soft-thin} in \cref{l:soft-repr-thin}, $\Xi$ is a functor between (large) preorders, i.e.\ a monotone map. We claim that $\Xi$ is an isomorphism of categories, i.e.\ an order isomorphism. To show that $\Xi$ is an order embedding, we must prove that there is an arrow $H\Rightarrow J$ in $\cat{M}$ whenever there is an arrow $\delta\colon (\gamma_* H, H(\top))\to (\gamma_* J, J(\top))$ in $\sKShr{A}(P,\cat{C})$. For all $p\in P$ we have a commutative diagram as displayed below.
\[\begin{tikzcd}
& A \arrow{dl}[swap]{H(\top)} \arrow{dr}{J(\top)} & \\
\gamma_* H(\top) \arrow{rr}{(\gamma_* \delta)_{\top}} \arrow{d}[swap]{\ros{\gamma_* H}{\top}{p}} & & \gamma_* J(\top) \arrow{d}{\ros{\gamma_* J}{\top}{p}} \\
\gamma_* H(p) \arrow{rr}{(\gamma_* \delta)_{p}} & & \gamma_* J(p)
\end{tikzcd}\]
Just observe that the triangle commutes because $\delta$ is a morphism of soft $\K$-sheaf representations, and the rectangle commutes by naturality of $\gamma_* \delta$. The composites 
\[
\ros{\gamma_* H}{\top}{p} \circ H(\top) \ \ \text{ and } \ \ \ros{\gamma_* J}{\top}{p} \circ J(\top)
\]
coincide with $H(p)$ and $J(p)$, respectively. This shows that $H(p)$ is below $J(p)$ for all $p\in P$, so there is an arrow $H\Rightarrow J$ in $\cat{M}$.

Next, we show that $\Xi$ is surjective, hence an order isomorphism. Fix an arbitrary $(F,\phi)\in \sKShr{A}(P,\cat{C})$ and consider the monotone map $H_F\colon P^\op\to\RegEpi{A}$ defined in \cref{p:sheaf-to-monot-map}. Then $H_F\in \cat{M}$ by \cref{HF-preservation} in \cref{p:sheaf-to-monot-map}, so we can consider its image $\Xi(H_F)=(\gamma_* H_F, H_F(\top))$. In view of \cref{gammastar-surj} in \cref{p:sheaf-to-monot-map} we have $\gamma_* H_F = F$. Moreover $H_F(\top)=\phi$, showing that $\Xi(H_F)=(F,\phi)$.
\end{proof}

The isomorphism of categories $\cat{M} \cong \sKShr{A}(P,\cat{C})$ in \cref{th:equiv-soft-sheaf-repr} can be equivalently understood as an order isomorphism between (large) preorders, which in turn induces an order isomorphism between the corresponding poset reflections. We state this observation in the following corollary in the special case of Barr-exact categories, where quotients can be replaced with equivalence relations. Let us write $\llbracket \sKShr{A}(P,\cat{C}) \rrbracket$ for the poset reflection of $\sKShr{A}(P,\cat{C})$; the objects of $\llbracket \sKShr{A}(P,\cat{C}) \rrbracket$ are isomorphism classes of soft $\K$-sheaf representations of $A$ over $P$.

\begin{corollary} \label{cor:MAIN}
Assume $\cat{C}$ is Barr-exact. 
Let $\cat{N}$ be the (large) sub-poset of $[P^\op,\Equiv{A}]$ consisting of those maps that preserve finite infima and arbitrary suprema, and whose images consist of pairwise commuting equivalence relations. There is an order isomorphism
\[
\cat{N} \cong \llbracket \sKShr{A}(P,\cat{C}) \rrbracket.
\]
\end{corollary}
\begin{proof}
It follows at once from \cref{th:equiv-soft-sheaf-repr} that there is an order isomorphism
\[
\llbracket \cat{M} \rrbracket \cong \llbracket \sKShr{A}(P,\cat{C}) \rrbracket
\]
between the poset reflections of $\cat{M}$ and $\sKShr{A}(P,\cat{C})$, respectively. So, it suffices to show that $\llbracket \cat{M} \rrbracket \cong \cat{N}$.
The poset reflection $\llbracket \RegEpi{A} \rrbracket$ of $\RegEpi{A}$ coincides with the opposite of $\Quo{A}$. If $\cat{C}$ is Barr-exact, the map $\ker\colon (\Quo{A})^\op\to \Equiv{A}$ is an order isomorphism (see \cref{rem:Barr-exact-ker-iso}), and so $\llbracket \RegEpi{A} \rrbracket$ is isomorphic to $\Equiv{A}$. Thus, the poset reflection of $[P^\op,\RegEpi{A}]$ is isomorphic to $[P^\op,\Equiv{A}]$. This isomorphism restricts to an isomorphism between $\llbracket \cat{M} \rrbracket$ and $\cat{N}$ because, under the isomorphism between $\llbracket \RegEpi{A} \rrbracket$ and $\Equiv{A}$, the (isomorphism classes of) ker-commuting elements of $\RegEpi{A}$ correspond to commuting elements of $\Equiv{A}$.
\end{proof}

\begin{remark}
If we required that $F(\bot)$ be a terminal object in the definition of $\K$-sheaf, then \cref{cor:MAIN} (as well as \cref{th:equiv-soft-sheaf-repr}) would no longer be true. Just observe that, if $\cat{C}=\Set$ and $A=\emptyset$, then both $\Equiv{A}$ and $[P^\op, \Equiv(A)]$ are one-element posets, and thus so is $\cat{N}$. By \cref{cor:MAIN}, there is exactly one (isomorphism class of) soft $\K$-sheaf representation of $A$ over $P$, given by the constant functor $F \colon P^\op \to \Set$ defined by $F(p)=\emptyset$ for all $p\in P$. But $F(\bot)=\emptyset$ is a proper subterminal object in $\Set$.
\end{remark}

\begin{remark}\label{rem:well-powered}
	Suppose that the category $\cat{C}$ is \emph{well-powered}, i.e.\ $\Sub{A}$ is a set---as opposed to a proper class---for all objects $A$ of $\cat{C}$. Then $\Equiv{A}$ is a small poset, and the isomorphism in \cref{cor:MAIN} is an isomorphism between small posets. This is the case, for example, when $\cat{C}$ is $\Set$ or a variety of (possibly infinitary) algebras.
\end{remark}

\section{Domains and Continuous Lattices}\label{s:domains}
In this brief interlude, we shall recall some basic material concerning domains and continuous lattices that will be needed in \cref{s:K-vs-Omega}. For a more thorough treatment, the reader can consult e.g.\ \cite{GierzHofmannKeimelEtAl2003}.

Let $P$ be a poset. The \emph{way-below} relation on $P$, denoted by $\ll$, is defined as follows: for all $x,y\in P$, $x\ll y$ precisely when, for all directed subsets $D\subseteq P$ admitting a supremum $\sup{D}$ in $P$, if $y \leq \sup{D}$ there is $d\in D$ such that $x \leq d$. 

\begin{remark}
	If $P$ is a complete lattice and $x,y\in P$, then $x\ll y$ if and only if, for all subsets $Y\subseteq P$, whenever $y\leq \sup{Y}$ there is a finite subset $X\subseteq Y$ such that $x\leq \sup{X}$.
\end{remark}

\begin{example}
Recall that a topological space is \emph{locally compact} if each of its points admits a compact neighbourhood.
Let $X$ be a locally compact space and consider its frame of opens $\Omega(X)$. For all $U,V\in\Omega(X)$, $U\ll V$ if and only if there is a compact subset $C\subseteq X$ such that $U\subseteq C \subseteq V$. 
\end{example}

\begin{definition}
For any poset $P$, we shall say that:
	\begin{itemize}[leftmargin=*]
		\item $P$ is a \emph{directed complete partially ordered set} (\emph{dcpo}, for short) if every directed subset $D\subseteq P$ has a supremum in $P$;
		\item $P$ is \emph{continuous} if, for all $x\in P$, the set \[\dd x \coloneqq \{y\in P\mid y\ll x\}\] is directed and $x = \sup{\dd x}$.
	\end{itemize}
	A dcpo that is continuous is called a \emph{domain}.
	A \emph{continuous lattice} is a domain that is complete as a lattice. 
\end{definition}

\begin{remark}\label{rem:fin-sup-dd-x}
	Whenever the poset $P$ has a least element $\bot$, we have $\bot \ll x$ for all $x\in P$.
	Moreover, for all $x,y,z\in P$ such that the supremum $x\vee y$ exists in $P$, $x\ll z$ and $y\ll z$ entail $x\vee y\ll z$, see e.g.\ \cite[Proposition~I-1.2(iii)]{GierzHofmannKeimelEtAl2003}.
	It follows that, whenever $P$ has all finite suprema, the set $\dd x$ is directed for all $x\in P$.
\end{remark}

\begin{example}\label{ex:X-locally-comp-well-filt}
	For a locally compact space $X$, its frame of opens $\Omega(X)$ is a continuous lattice. See e.g.\ \cite[p.~56]{GierzHofmannKeimelEtAl2003}.
\end{example}

\begin{definition}
	Let $U$ be an upwards closed subset of a dcpo $P$. We say that $U$ is \emph{Scott-open} if, for all directed subsets $D\subseteq P$, $\sup{D}\in U$ entails $D \cap U \neq \emptyset$. 
\end{definition}

Recall that a \emph{filter} on a poset $P$ is an upwards closed subset of $P$ that is codirected (in particular, non-empty). Write \[\Filt(P)\] for the poset of filters on $P$, partially ordered by inclusion. $\Filt(P)$ is a dcpo, and is a complete lattice whenever $P$ has binary suprema and a top element. 
A filter on a dcpo that is Scott-open as an upwards closed subset is called a \emph{Scott-open filter}. The set of all Scott-open filters on a dcpo $P$ is denoted by \[\Scott(P)\] and regarded as a sub-poset of $\Filt(P)$. The poset $\Scott(P)$ is often referred to as the \emph{Lawson dual} of $P$; it is always a dcpo, with directed suprema given by set-theoretic unions, and is a domain whenever $P$ is a domain (see e.g.\ \cite[Theorem~II-1.17]{GierzHofmannKeimelEtAl2003}). Further, if $P$ admits finite infima, $\Scott(P)$ has finite infima given by set-theoretic intersections.

We now recall some basic properties of Scott-open filters on a domain. To this end, for any element $x$ of a poset $P$, set $\u x \coloneqq \{y \in P \mid x \leq y\}$.
\begin{lemma}\label{l:Scott-filters-prop}
	Let $L$ be a domain. The following statements hold:
	\begin{enumerate}[label=(\alph*)]
		\item\label{L-way-below-Scott} For all $x,y\in L$, $y \ll x$ precisely when there is $k\in\Scott(L)$ such that $x\in k \subseteq \u y$.
		\item\label{Scott-round-filter} For all $x\in L$ and all $k\in \Scott(L)$, if $x\in k$ there is $y\ll x$ such that $y\in k$.
		\item\label{Scott-way-below} For all $k,l\in\Scott(L)$, $l\ll k$ precisely when there is $x\in k$ such that $l\subseteq \u x$.
		\item\label{Scott-round-filter-cons} For all $x\in L$ and all $k\in \Scott(L)$, if $x\in k$ there is $l\ll k$ such that $x\in l$.
		\item\label{U(s)-codirected} For all $x \in L$, the set $\{k \in \Scott(L) \mid x \in k\}$ is codirected.
	\end{enumerate} 
\end{lemma}
\begin{proof}
	The left-to-right implication in \cref{L-way-below-Scott} follows from \cite[Proposition~I-3.3(i)]{GierzHofmannKeimelEtAl2003}, while the converse is a consequence of the definition of Scott-open filter.
	For \cref{Scott-round-filter,Scott-way-below}, see e.g.\ Proposition~II-1.10(i) and Theorem~II-1.17(ii), respectively, in \cite{GierzHofmannKeimelEtAl2003}.
	
	For \cref{Scott-round-filter-cons}, let $x\in L$ and $k\in \Scott(L)$ satisfy $x\in k$. By \cref{Scott-round-filter}, there is $y\ll x$ such that $y\in k$. Thus, in view of \cref{L-way-below-Scott}, there exists $l\in\Scott(L)$ such that $x\in l\subseteq \u y$. Since $y \in k$ and $l \subseteq \u y$, by \cref{Scott-way-below} we get $l\ll k$.

Finally, let us prove \cref{U(s)-codirected}. Let $S$ be a finite subset of $\{k \in \Scott(L) \mid x \in k\}$. For each $l \in S$, pick $y_l \in l$ such that $y_l \ll x$ (existence is guaranteed by \cref{Scott-round-filter}). Since $L$ is a domain, the set $\dd x$ is directed and thus there is $y \in \dd x$ such that $y_l \leq y$ for all $l \in S$. By \cref{L-way-below-Scott}, there is $k \in \Scott(L)$ such that $x \in k \subseteq \u y$. For every $l \in S$ we have $k \subseteq \u y \subseteq \u y_l \subseteq l$. Therefore, $k$ is a lower bound of $S$.
\end{proof}

\begin{remark} \label{r:U(s)-open-filter}
\Cref{Scott-round-filter-cons,U(s)-codirected} in \cref{l:Scott-filters-prop} can be strengthened to the effect that, for every element $x$ of a domain $L$, the set $\mathcal{U}(x) \coloneqq \{k \in \Scott(L) \mid x \in k\}$ is a Scott-open filter on $\Scott(L)$. In fact, the Scott-open filters on $\Scott(L)$ are precisely those of the form $\mathcal{U}(x)$, for $x \in L$, and the map $\mathcal{U}(-) \colon L \to \Scott(\Scott(L))$ is an order isomorphism. These observations are at the base of the Lawson Duality Theorem of Domains \cite[Theorem~IV-2.14]{GierzHofmannKeimelEtAl2003}. For a proof of the properties mentioned above, cf.\ \cite[\S IV-2]{GierzHofmannKeimelEtAl2003}.
\end{remark}

Given a subset $S$ of a poset $L$, set $\dd S \coloneqq \bigcup{\{\dd x \mid x \in S\}}$.

\begin{lemma} \label{l:downset}
	Let $L$ be a domain.
	For all directed subsets $D\subseteq L$, $\dd D = \dd \sup{D}$.
\end{lemma}
\begin{proof}
	The inclusion $\dd D \subseteq \dd \sup{D}$ is immediate because $x \ll y \leq z$ implies $x \ll z$. For the converse inclusion just recall that, in a continuous poset, $x \ll \sup{D}$ implies the existence of $d\in D$ such that $x \ll d$, see e.g.\ \cite[Theorem~I-1.9]{GierzHofmannKeimelEtAl2003}.
\end{proof}

We record for future reference the following fact, which is known as \emph{Wilker's condition} when $L=\Omega(X)$ for a locally compact space $X$ (cf.~\cite{KL2005} or \cite[Lemma~2.3]{GehrkeGool2018}).
\begin{lemma}\label{l:wilker}
Let $L$ be a continuous lattice.
For all $x,y\in L$ and $l\in\Scott(L)$ such that $x\vee y\in l$, there are $k,k'\in\Scott(L)$ such that $x\in k$, $y\in k'$ and $k\wedge k'\subseteq l$. 
\end{lemma}
\begin{proof}
Let $J\coloneqq \dd x \cup \dd y$. Since $\dd x \subseteq J$ and $L$ is continuous, we get
\[
x = \sup{\dd x} \leq \sup{J}.
\]
Similarly, $y\leq \sup{J}$ and thus $x\vee y \leq \sup{J}$. As $l$ is upwards closed, $\sup{J}\in l$. Now, because $l$ is a Scott-open filter, there exists a finite subset $I \subseteq J$ such that $\sup I \in l$; since $\dd x$ and $\dd y$ are closed under finite suprema by \cref{rem:fin-sup-dd-x}, there are $x' \in \dd x$ and $y' \in \dd y$ such that 
\[
x' \vee y' = \sup I \in l.
\] 

In view of \cref{L-way-below-Scott} in \cref{l:Scott-filters-prop}, there are $k,k'\in \Scott(L)$ such that 
\[
x\in k\subseteq \u x' \ \ \text{ and } \ \ y\in k'\subseteq \u y'.
\] 
The Scott-open filters $k,k'$ then satisfy the desired property. Just observe that, if $z\in L$ belongs to both $k$ and $k'$, then $x'\le z$ and $y'\le z$ entail $x'\vee y' \le z$. As $x'\vee y'\in l$ and the latter is upwards closed, we get $z\in l$.
\end{proof}

Given a topological space~$X$, let $\K(X)$ denote the poset of compact saturated subsets of~$X$, ordered by inclusion (recall that a set is \emph{saturated} if it is an intersection of opens).
The Hofmann--Mislove theorem states that, for any \emph{sober} space $X$ (i.e., one in which every irreducible closed subset is the closure of a unique point), the monotone map
\[
\Phi\colon \K(X)^\op \to \Scott(\Omega(X)), \ \ \Phi(K)\coloneqq \{U\in \Omega(X) \mid K\subseteq U\}
\]
is an order isomorphism. Its inverse sends a Scott-open filter of open sets to its intersection. See e.g.\ \cite[Theorem~II-1.20]{GierzHofmannKeimelEtAl2003}.

Even when $P$ is a continuous lattice, $\Scott(P)$ need not be complete (equivalently, a continuous lattice) because it may fail to admit binary suprema. Thus, we shall now focus on a class of continuous lattices whose posets of Scott-open filters are complete. 
\begin{definition}
The way-below relation $\ll$ in a continuous lattice $L$ with top element $\top$ is said to be \emph{multiplicative} provided that $\top\ll \top$ and, for all $x,y,z\in L$, $x\ll y$ and $x\ll z$ entail $x\ll y\wedge z$.
\end{definition}
A continuous lattice whose way-below relation is multiplicative is called a \emph{stably continuous lattice}. If $L$ is a stably continuous lattice then so is $\Scott(L)$, see e.g.\ \cite[\S VII.2.12]{Johnstone1986}. In this case, the supremum of Scott-open filters $U,V$ is given by $\u \{u\wedge v\mid u\in U, \, v\in V\}$.

\begin{example}\label{ex:stably-compact}
For a $T_0$ space $X$, its frame of opens $\Omega(X)$ is a stably continuous lattice if and only if $X$ is \emph{stably compact}, i.e.\ $T_0$, compact, locally compact, coherent\footnote{A topological space is \emph{coherent} if the intersection of any two compact saturated subsets is compact.} and sober. See e.g.\ \cite[Proposition~VI-7.3]{GierzHofmannKeimelEtAl2003}.
Stably compact spaces generalise both compact Hausdorff spaces and spectral spaces, and are tightly related to compact ordered spaces, cf.\ \cref{s:CompOrd}.
\end{example}

\section{$\K$-sheaves and \texorpdfstring{$\Omega$}{\unichar{"03A9}}-sheaves}\label{s:K-vs-Omega}
Throughout this section, we fix an arbitrary domain $L$. There are order embeddings
\[\begin{tikzcd}
L^\op \arrow[rightarrowtail]{r}{\lambda} & \Filt(L) & \Scott(L) \arrow[rightarrowtail]{l}[swap]{\kappa}
, \end{tikzcd}\]
where $\lambda$ sends $x\in L$ to the principal filter $\u x$, and $\kappa$ is the inclusion of Scott-open filters on $L$ into filters on $L$. Let $\F$ be the union of the images of $\lambda$ and $\kappa$. We regard $\F$ as a sub-poset of $\Filt(L)$ and, with a slight abuse of notation, write again $\lambda\colon L^\op \emb \F$ and $\kappa\colon \Scott(L) \emb \F$ for the obvious co-restrictions. The latter order embeddings will be regarded as functors between small (posetal) categories.

In the first part of this section we only assume that $\cat{C}$ is a \emph{bicomplete} (i.e., complete and cocomplete) category. The functors $\lambda,\kappa$ induce ``restriction'' functors
\begin{center}
\begin{tikzcd} {[\F, \cat{C}]} \arrow{r}{\lambda^*} & {[L^\op, \cat{C}]} \end{tikzcd}
 \ \ and \ \ 
\begin{tikzcd} {[\F, \cat{C}]} \arrow{r}{\kappa^*} & {[\Scott(L), \cat{C}]} \end{tikzcd}
\end{center}
given by precomposing with $\lambda$ and $\kappa$, respectively. As $L^\op$ and $\Scott(L)$ are small categories and $\cat{C}$ is bicomplete, any functor $G$ in $[L^\op, \cat{C}]$ admits a left Kan extension $\Lan_{\lambda} G$ along $\lambda$, and any functor $F$ in $[\Scott(L), \cat{C}]$ admits a right Kan extension $\Ran_{\kappa} F$ along $\kappa$. These Kan extensions are computed pointwise and determine two adjunctions as displayed below (see e.g.\ \cite[\S X.3]{MacLane1998}):
\[\begin{tikzcd}[column sep = 4em]
{[\F, \cat{C}]} \arrow[yshift=4pt]{r}{\lambda^*} \arrow[r,phantom,"\text{\tiny{$\top$}}"] & {[L^\op, \cat{C}]} \arrow[yshift=-4pt]{l}{\Lan_{\lambda}} \\ 
{[\F, \cat{C}]} \arrow[yshift=-4pt]{r}[swap]{\kappa^*} \arrow[r,phantom,"\text{\tiny{$\top$}}"] & {[\Scott(L), \cat{C}].} \arrow[yshift=4pt]{l}[swap]{\Ran_{\kappa}}
\end{tikzcd}\]
Let $\eta,\epsilon$ be the unit and counit, respectively, of the adjunction $\Lan_{\lambda}\dashv \lambda^*$, and $\tilde{\eta},\tilde{\epsilon}$ the unit and counit, respectively, of the adjunction $\kappa^*\dashv \Ran_{\kappa}$.

As $\lambda$ and $\kappa$ are order embeddings (hence, fully faithful), these adjunctions fix all objects of $[L^\op, \cat{C}]$ and $[\Scott(L), \cat{C}]$, respectively. This is the content of the next lemma, which holds more generally for Kan extensions along fully faithful functors, cf.\  \cite[Proposition~4.23]{Kelly1982}.

\begin{lemma}\label{l:fixed-side-Kan}
The following statements hold:
\begin{enumerate}[label=(\alph*)]
\item\label{second-adj} The unit $\eta$ of the adjunction $\Lan_{\lambda}\dashv \lambda^*$ is a natural isomorphism, i.e.\ for all $G\in [L^\op, \cat{C}]$, $\eta_G\colon G\cong \lambda^* \Lan_{\lambda} G$.
\item\label{first-adj} The counit $\tilde{\epsilon}$ of the adjunction $\kappa^*\dashv \Ran_{\kappa}$ is a natural isomorphism, i.e.\ for all $F\in [\Scott(L), \cat{C}]$, $\tilde{\epsilon}_F\colon \kappa^* \Ran_{\kappa}  F\cong F$.
\end{enumerate}
\end{lemma}

\begin{remark}\label{rem:Lan-Ran-ff}
Recall that, given any adjunction, the right adjoint is fully faithful precisely when the counit is a natural isomorphism. A dual statement holds for left adjoint functors and units. Therefore, \cref{l:fixed-side-Kan} amounts to saying that the functors $\Ran_{\kappa}$ and $\Lan_{\lambda}$ are fully faithful.
\end{remark}

Composing the two adjunctions above, we obtain the following pair of adjoint functors:
\begin{equation} \label{e:adjunction}
	\begin{tikzcd}[column sep = 4em]
		{[\Scott(L), \cat{C}]} \arrow[yshift=4pt]{r}{\lambda^*\Ran_{\kappa}} \arrow[r,phantom,"\text{\tiny{$\top$}}"] & {[L^\op, \cat{C}].} \arrow[yshift=-4pt]{l}{\kappa^*\Lan_{\lambda}}
	\end{tikzcd}
\end{equation}
In \cref{p:directed-sheaves} below we shall characterise the objects that are fixed by the unit and counit of this adjunction, thus obtaining an equivalence between the full subcategories defined by the fixed objects.
By further restricting this equivalence, in \cref{equiv-Om-K-sheaves} we will relate the notion of $\K$-sheaf to the usual notion of sheaf.

\begin{remark} \label{r:explicitly}
	We provide explicit descriptions of the adjoint functors in~\cref{e:adjunction}. These are easy consequences of the formulas for pointwise Kan extensions, cf.\ \cite[\S X.5]{MacLane1998}.
	For all $G \colon L^\op \to \cat{C}$ and $k \in \Scott(L)$, 
	\[
	(\kappa^*\Lan_{\lambda}G)(k)
	\] is the directed colimit in $\cat{C}$ of the restriction of $G$ to $k$ (regarded as a subset of $L$).
	If $k, l \in \Scott(L)$ are such that $k \subseteq l$, the corresponding arrow 
	\[
	(\kappa^*\Lan_{\lambda}G)(k) \to (\kappa^*\Lan_{\lambda}G)(l)
	\]
	 is the mediating morphism induced by the universal (colimit) property of $(\kappa^*\Lan_{\lambda}G)(k)$.
	
	On the other hand, for all $F \colon \Scott(L) \to \cat{C}$ and $x \in L$, 
	\[
	(\lambda^*\Ran_{\kappa}F)(x)
	\] 
	is the codirected limit in $\cat{C}$ of the restriction of $F$ to $\{k \in \Scott(L) \mid x \in k\}$. Note that the latter set is codirected by \cref{U(s)-codirected} in \cref{l:Scott-filters-prop}.
	Given $x, y \in L$ with $x \leq y$, the corresponding morphism 
	\[
	(\lambda^*\Ran_{\kappa}F)(y) \to (\lambda^*\Ran_{\kappa}F)(x)
	\] 
	is the one induced by the universal (limit) property of $(\lambda^*\Ran_{\kappa}F)(x)$.
\end{remark}

\begin{remark}\label{rem:K-sheaf-encodes-stalks}
Assume that $X$ is a sober space. By the Hofmann--Mislove theorem, we have an isomorphism $\Scott(\Omega(X))\cong\K(X)^\op$. Hence, for any sheaf of sets $G\colon \Omega(X)^\op\to \Set$ over $X$, $\kappa^*\Lan_{\lambda}G$ can be identified with a presheaf $\K(X)^\op\to\Set$. For any point $x\in X$, the up-set $\u{x}$ of $x$ in the specialization order of $X$ is a compact saturated subset of $X$. In view of \cref{r:explicitly}, the value of $\kappa^*\Lan_{\lambda}G$ at $\u{x}$ is
\[
(\kappa^*\Lan_{\lambda}G)(\u{x}) \cong \colim_{\u{x}\subseteq U\in\Omega(X)}{G(U)} \cong \colim_{x\in U\in\Omega(X)}{G(U)},
\]
which is precisely the stalk of $G$ at $x$. A similar remark applies when $\Set$ is replaced with any variety of algebras.
\end{remark}

\begin{lemma}\label{l:K3-Om3}
Consider a functor $F \colon L \to \cat{C}$. The following statements are equivalent:
	\begin{enumerate}
		\item \label{i:dir-colim-1}
		$F$ preserves directed colimits.
		\item \label{i:dir-colim-2}
		For all $x \in L$, the cocone $(\ros{F}{y}{x}\colon F(y) \to F(x))_{y \in \dd x}$ is a (directed) colimit of the restriction of $F$ to $\dd x$.
	\end{enumerate}
\end{lemma}
\begin{proof}
	Because $L$ is continuous, for all $x\in L$ the set $\dd x$ is directed and $x = \sup{\dd x}$.
	Hence, \ref{i:dir-colim-2} is an immediate consequence of \ref{i:dir-colim-1}.
	
	Conversely, suppose that \ref{i:dir-colim-2} holds and fix an arbitrary directed set $\D\subseteq L$. We must show that the cocone 
	\[
	(\ros{F}{y}{\sup \D}\colon F(y) \to F\big(\sup \D\big))_{y \in \D}
	\] 
	is a colimit of the restriction of $F$ to~$\D$. Let $(\phi_y\colon F(y)\to A)_{y\in \D}$ be a compatible cocone over the restriction of $F$ to~$\D$.
	For each $z \in \dd \D$, let $\psi_z \colon F(z) \to A$ denote the composite $\phi_{y_z}\circ \ros{F}{z}{y_z}$, where $y_z$ is any element of $\D$ such that $z \ll y_z$: since $\D$ is directed and the cocone $(\phi_y \colon F(y) \to A)_{y \in \D}$ is compatible, $\psi_z$ does not depend on the choice of $y_z$.
	It is not difficult to see that $(\psi_z \colon F(z) \to A)_{z \in \dd\D}$ is a compatible cocone over the restriction of $F$ to $\dd\D$. As $\dd \D = \dd \sup{\D}$ by \cref{l:downset}, it follows from \ref{i:dir-colim-2} that there is a unique arrow $\eta \colon F(\sup\D) \to A$ such that $\psi_z = \eta\circ \ros{F}{z}{\sup{\D}}$ for all $z \in \dd\D$.
	
	We claim that $\eta$ is the unique arrow such that $\phi_y = \eta\circ \ros{F}{y}{\sup{\D}}$ for all $y\in \D$, thus showing that $(\ros{F}{y}{\sup \D}\colon F(y) \to F(\sup \D))_{y \in \D}$ is a colimit cocone.
	To this end, fix an arbitrary $y\in\D$.
	Applying \ref{i:dir-colim-2} again, we have that the cocone $(\ros{F}{z}{y}\colon F(z) \to F(y))_{z \in \dd y}$ is a colimit of the restriction of $F$ to $\dd y$.
	Using the universal property of the latter colimit, we see that $\phi_y$ must coincide with the composite 
	\[\begin{tikzcd}[column sep = 3em]
		F(y) \arrow{r}{\ros{F}{y}{\sup{\D}}} & F(\sup{\D}) \arrow{r}{\eta} & A.
	\end{tikzcd}\]
	Just observe that, for all $z\in \dd y$,
	\[
	\phi_y \circ \ros{F}{z}{y} = \phi_{y_z}\circ \ros{F}{z}{y_z} = \psi_z = \eta\circ \ros{F}{z}{ \sup{\D}}.
	\]
	Clearly, $\eta$ is unique with respect to this property.
\end{proof}

It is useful to record the dual version of the previous lemma:
\begin{lemma} \label{l:Om3-K3}
Consider a functor $F \colon L^\op \to \cat{C}$. The following statements are equivalent:
	\begin{enumerate}
		\item
		$F$ preserves codirected limits.
		\item
		For all $x \in L$, the cone $(\ros{F}{x}{y}\colon F(x) \to F(y))_{y \in \dd x}$ is a (codirected) limit of the restriction of $F$ to $\dd x$.
	\end{enumerate}
\end{lemma}

\begin{proof}
	This follows by applying \cref{l:K3-Om3} to the functor $F^\op \colon L \to \cat{C}^\op$.
\end{proof}

Now, denote by 
\[
\omega\text{-}\lim[L^\op, \cat{C}]
\] 
the full subcategory of $[L^\op, \cat{C}]$ defined by the functors preserving codirected limits, and~by 
\[
\omega\text{-}\colim[\Scott(L), \cat{C}]
\] 
the full subcategory of $[\Scott(L), \cat{C}]$ defined by the functors preserving directed colimits.

\begin{lemma} \label{l:composition}
	The following statements hold:
	\begin{enumerate}[label=(\alph*)]
		\item \label{i:kappa-Lan}
		For all $G \in [L^\op, \cat{C}]$, the functor $\kappa^*\Lan_{\lambda}G$ belongs to $\omega\text{-}\colim[\Scott(L), \cat{C}]$.
		\item \label{i:lambda-Ran}
		For all $F \in [\Scott(L), \cat{C}]$, the functor $\lambda^*\Ran_{\kappa}F$ belongs to $\omega\text{-}\lim[L^\op, \cat{C}]$.
	\end{enumerate}
\end{lemma}

\begin{proof}
For \cref{i:kappa-Lan}, fix an arbitrary $k\in \Scott(L)$.	
	In view of \cref{l:K3-Om3}, it suffices to prove that the cocone $(\ros{G}{l}{k}\colon(\kappa^* \Lan_\lambda G)(l) \to (\kappa^* \Lan_\lambda G)(k))_{l \in \dd k}$ is a colimit of the restriction of $\kappa^*\Lan_\lambda G $ to $\dd k$.
	Note that, since $\Scott(L)$ is continuous, we have $k = \bigcup{\dd k}$.
	Therefore, using \cref{r:explicitly} twice, we have
	\[
		(\kappa^*\Lan_\lambda G)(k) \cong \colim_{x \in k} G(x) \cong \colim_{l \ll k,\, x \in l} G(x) \cong \colim_{l \ll k}\colim_{x \in l}G(x) \cong \colim_{l \ll k} (\kappa^* \Lan_\lambda G)(l).
	\]
This slick proof does not show that the cocone above is a colimit cocone, although this could be deduced by checking how the colimit cocones are modified under the chain of isomorphisms. We now give a detailed proof that follows precisely this intuition.

To start with, observe that $\kappa^*$ is left adjoint and so it preserves colimits. Thus, it is enough to show that the cocone 
\begin{equation}\label{eq:claim-colimit-Lan-G}
(\ros{G}{l}{k}\colon \Lan_{\lambda}G(l) \to \Lan_{\lambda}G(k))_{l \in \dd k}
\end{equation}
is a colimit of the restriction of $\Lan_{\lambda} G$ to $\dd k$.
Let $(\phi_l\colon \Lan_{\lambda} G(l) \to A)_{l\in \dd k}$ be a compatible cocone over the restriction of $\Lan_{\lambda} G$ to $\dd k$. 
Because $k = \bigcup{\dd k}$, for all $x\in k$ there is $l\in \Scott(L)$ such that $l\ll k$ and $x\in l$.
Hence, for each $x\in k$, we obtain an arrow 
\[
\phi_x\colon \Lan_{\lambda} G(\u x)\to A
\] 
by composing $\phi_l\colon \Lan_{\lambda} G(l)\to A$ with $\ros{G}{\u x}{l}\colon \Lan_{\lambda} G(\u x)\to \Lan_{\lambda} G(l)$. Note that the definition of $\phi_x$ does not depend on the choice of $l$ because $\dd k$ is directed, and 
the family $\{\phi_x\colon \Lan_{\lambda} G(\u x)\to A\mid x\in k\}$ forms a compatible cocone over the restriction of $\Lan_{\lambda} G$ to~$k$. The colimit of the latter diagram is $\Lan_{\lambda} G(k)$, so there is a unique mediating morphism $\mu\colon \Lan_{\lambda} G(k)\to A$ such that $\phi_x= \mu\circ \ros{G}{\u x}{k}$ for all $x\in k$. We claim that $\phi_l = \mu\circ \ros{G}{l}{k}$ for all $l\ll k$. Fix an arbitrary Scott-open filter $l\in\dd k$. By \cref{Scott-way-below} in \cref{l:Scott-filters-prop}, there is $x\in k$ such that $l\subseteq \u x$ and thus $\ros{G}{l}{k}$ factors through $\ros{G}{\u x}{k}$. We get
\[
\mu\circ \ros{G}{l}{k}= \mu\circ \ros{G}{\u x}{k}\circ \ros{G}{l}{\u x}= \phi_x\circ \ros{G}{l}{\u x},
\]
which in turn coincides with $\phi_l$ by construction. Further, it is not difficult to see that $\mu$ is the unique morphism satisfying $\phi_l = \mu\circ \ros{G}{l}{k}$ for all $l\ll k$. Hence, the cocone in \cref{eq:claim-colimit-Lan-G} is a colimit of the restriction of $\Lan_{\lambda} G$ to $\dd k$.

Now, for \cref{i:lambda-Ran}, fix an arbitrary $x \in L$. In view of \cref{l:Om3-K3}, it suffices to prove that the cone $(\ros{F}{x}{y}\colon (\lambda^*\Ran_{\kappa} F)(x) \to (\lambda^*\Ran_{\kappa} F)(y))_{y \in \dd x}$ is a limit of the restriction of $\lambda^*\Ran_{\kappa} F$ to $\dd x$.
	By \cref{Scott-round-filter} in \cref{l:Scott-filters-prop}, we have
	\[
		\{l \in \Scott(L) \mid x \in l\} = \{ l \in \Scott(L) \mid \exists y \in L \text{ such that } y \ll x \text{ and } y \in l \}.
	\]
	Using \cref{r:explicitly} twice, we get
	\[
		(\lambda^*\Ran_{\kappa} F)(x) \cong \lim_{x \in l}F(l) = \lim_{y \ll x} \lim_{y \in l} F(l) = \lim_{y \ll x} (\lambda^*\Ran_{\kappa} F)(y),
	\]
	where $l\in\Scott(L)$. 
	Reasoning along the same lines as before, it is possible to give a direct proof that the cone above is a limit cone; we leave the details to the reader.
\end{proof}

As a consequence of \cref{l:composition}, every object of $[L^\op, \cat{C}]$ that is fixed by the unit of the adjunction in~\cref{e:adjunction} belongs to $\omega\text{-}\lim[L^\op, \cat{C}]$, and every object of $[\Scott(L), \cat{C}]$ that is fixed by the counit belongs to $\omega\text{-}\colim[\Scott(L), \cat{C}]$.
The next lemma will imply that the converse implications hold as well. 

Recall that $\tilde{\eta}$ is the unit of the adjunction $\kappa^*\dashv \Ran_{\kappa}$ and $\epsilon$ is the counit of ${\Lan_{\lambda}\dashv \lambda^*}$. 
Intuitively, for all functors $F\colon \mathcal{F}\to\cat{C}$, $\tilde{\eta}_F$ is an isomorphism precisely when, if we first restrict $F$ along $\kappa^*$ and then extend along $\Ran_{\kappa}$, we recover $F$ up to a natural isomorphism. A similar remark applies to the case where $\epsilon_F$ is an isomorphism.

\begin{lemma} \label{l:directed-sheaves-F}
	The following conditions are equivalent for all functors $F\in [\mathcal{F}, \cat{C}]$:
	\begin{enumerate}
		\item \label{i:K-like} 
		$\kappa^* F\in \omega\text{-}\colim[\Scott(L), \cat{C}]$ and $\tilde{\eta}_F$ is an isomorphism.
		\item \label{i:Om-like}
		$\lambda^* F\in \omega\text{-}\lim[L^\op, \cat{C}]$ and $\epsilon_F$ is an isomorphism.	
	\end{enumerate}
\end{lemma}

\begin{proof}
	Let us prove that \ref{i:K-like} implies \ref{i:Om-like}. Fix an arbitrary functor $F\colon \mathcal{F}\to \cat{C}$ such that $\kappa^* F$ preserves directed colimits and $\tilde{\eta}_F$ is an isomorphism. 
	By \cref{i:lambda-Ran} in \cref{l:composition} applied to $\kappa^* F$, we have that $\lambda^* \Ran_{\kappa} \kappa^* F$ preserves codirected limits. In view of the isomorphism $\lambda^*(\tilde{\eta}_F)\colon \lambda^* F\cong \lambda^* \Ran_{\kappa}\kappa^* F$, we get $\lambda^* F\in \omega\text{-}\lim[L^\op, \cat{C}]$.
	
	Next, we show that $\epsilon_F$ is an isomorphism. 
	That is, for all $\phi\in\F$, 
	\[
	(\epsilon_F)_{\phi}\colon (\Lan_{\lambda}\lambda^* F)(\phi) \to F(\phi)
	\]
	is an isomorphism.
	By the formula for pointwise left Kan extensions (cf.\ \cite[\S X.5]{MacLane1998}), $(\epsilon_F)_{\phi}$ can be identified with the unique mediating morphism
	\[
	\colim_{\u x \subseteq \phi}F(\u x) \to F(\phi)
	\]
	for $x\in L$. If $\phi$ belongs to the image of $\lambda$ (i.e., $\phi$ is a principal filter) then this mediating morphism is clearly an isomorphism. Hence it remains to prove that, for all $k\in\Scott(L)$, the following arrow is an isomorphism:
	\[
	\colim_{x\in k} F(\u x) \to F(k).
	\]
	Let $\{\psi_x\colon F(\u x)\to A\mid x\in k\}$ be a compatible cocone in $\cat{C}$ over the restriction of $F$ to $\{\u x \mid x\in k\}$. By \cref{Scott-way-below} in \cref{l:Scott-filters-prop}, whenever two Scott-open filters $k$ and $k'$ satisfy $k'\ll k$, there is $x\in k$ such that $k'\subseteq \u x$. Hence, for every $k' \in \dd k$ we obtain a morphism 
	\[
	\phi_{k'}\colon F(k')\to A
	\] 
	by composing $\ros{F}{k'}{\u x}\colon F(k')\to F(\u x)$ with $\psi_x\colon F(\u x)\to A$. Note that the definition of $\phi_{k'}$ does not depend on the choice of $x$ because $k$ is codirected, and $\{\phi_{k'}\colon F(k')\to A\mid k'\ll k\}$ is a compatible cocone over the restriction of $\kappa^* F$ to $\dd k$ (recall that, for all $l\in \Scott(L)$, $\kappa^* F(l)=F(l)$). 
	As $\kappa^* F$ preserves directed colimits, we have
	\[
	\kappa^* F(k)\cong \colim_{k'\ll k} \kappa^*F(k').
	\]
	More precisely, the cocone $(\ros{F}{k'}{k}\colon \kappa^* F(k') \to \kappa^* F(k))_{k' \in \dd k}$ is a colimit of the restriction of $\kappa^* F$ to $\dd k$. 
	Thus, there is a unique mediating morphism $\mu\colon F(k)\to A$ satisfying $\phi_{k'}=\mu\circ \ros{F}{k'}{k}$ for all $k'\ll k$. It remains to show that $\mu$ satisfies $\psi_x=\mu\circ \ros{F}{\u x}{k}$ for all $x\in k$ and is unique with this property. 
	Fix an arbitrary $x\in k$. By \cref{Scott-round-filter-cons} in \cref{l:Scott-filters-prop} there is $k'\ll k$ such that $x\in k'$. It follows that $\ros{F}{\u x}{k}=\ros{F}{k'}{k}\circ \ros{F}{\u x}{k'}$ and so
	\[
	\mu\circ \ros{F}{\u x}{k} = \mu \circ \ros{F}{k'}{k}\circ \ros{F}{\u x}{k'} = \phi_{k'} \circ \ros{F}{\u x}{k'},
	\]
	which coincides with $\psi_x$. To deduce that $\mu$ is unique with this property, it is enough to note that any $\nu$ satisfying $\psi_x=\nu\circ \ros{F}{\u x}{k}$ for all $x\in k$ must satisfy $\phi_{k'}=\nu\circ \ros{F}{k'}{k}$ for all $k'\ll k$. Just observe that, as mentioned above, $k'\ll k$ entails the existence of $x\in k$ such that $k'\subseteq \u x$. Thus,
	\[
	\nu\circ \ros{F}{k'}{k} = \nu \circ \ros{F}{\u x}{ k} \circ \ros{F}{k'}{ \u x} = \psi_x \circ \ros{F}{k'}{ \u x} = \phi_{k'}.
	\]
	
	To show that \ref{i:Om-like} implies \ref{i:K-like}, suppose that $F\colon \mathcal{F}\to\cat{C}$ is such that $\lambda^* F$ preserves codirected limits and $\epsilon_F$ is an isomorphism. \Cref{i:kappa-Lan} in \cref{l:composition} applied to $\lambda^* F$ shows that $\kappa^*\Lan_{\lambda}\lambda^* F$ preserves directed colimits. Because $\kappa^*\Lan_{\lambda}\lambda^* F$ is naturally isomorphic to $\kappa^* F$ via $\kappa^*(\epsilon_F)$, we see that $\kappa^* F\in \omega\text{-}\colim[\Scott(L), \cat{C}]$.
	
	It remains to prove that $\tilde{\eta}_F$ is an isomorphism. That is, for all $\phi\in\F$, 
	\[
	(\tilde{\eta}_F)_{\phi}\colon F(\phi) \to (\Ran_{\kappa}\kappa^* F)(\phi)
	\]
	is an isomorphism.
	In view of the formula for pointwise right Kan extensions, $(\tilde{\eta}_F)_{\phi}$ can be identified with the unique mediating morphism
	\[
	F(\phi) \to \lim_{\phi\subseteq k}{F(k)}
	\]
	for $k$ which ranges over $\Scott(L)$. If $\phi$ belongs to the image of $\kappa$ (i.e., $\phi$ is a Scott-open filter) then this mediating morphism is clearly an isomorphism. Hence it suffices to prove that, for all $x\in L$, the following arrow is an isomorphism:
	\[
	F(\u x) \to \lim_{x\in k}{F(k)}.
	\]
	Let $\{\psi_k\colon A\to F(k)\mid x\in k\}$ be a compatible cone in $\cat{C}$ over the restriction of $F$ to $\{k\in\Scott(L) \mid x\in k\}$. Fix an arbitrary $y\ll x$. By \cref{L-way-below-Scott} in \cref{l:Scott-filters-prop}, there is a Scott-open filter $k$ such that $x\in k\subseteq \u y$. Hence we can define an arrow 
	\[
	\phi_y\colon A \to F(\u y)
	\] 
	as the composition of $\psi_k\colon A\to F(k)$ and the restriction map $\ros{F}{k}{\u y} \colon F(k)\to F(\u y)$.
	Note that the definition of $\phi_y$ does not depend on the choice of $k$ because $\dd x$ is directed, and the family $\{\phi_y \colon A \to F(\u y) \mid y\ll x\}$ forms a compatible cone over the restriction of $F$ to $\{\u y \mid y \in \dd x\}$. As $\lambda^* F$ preserves codirected limits, we have $\lambda^* F(x) \cong \lim_{y\ll x}{\lambda^* F(y)}$ and so $F(\u x) \cong \lim_{y\ll x}{F(\u y)}$. Thus the compatible cone above induces a unique arrow $\mu\colon A \to F(\u x)$ such that $\phi_y = \ros{F}{\u x}{ \u y} \circ \mu$ for all $y\ll x$. It remains to prove that $\mu$ satisfies $\psi_k = \ros{F}{\u x}{ k} \circ \mu$ for all Scott-open filters $k$ containing $x$ and is unique with this property. Fix an arbitrary $k\in\Scott(L)$ such that $x\in k$. By \cref{Scott-round-filter} in \cref{l:Scott-filters-prop} there is $y\ll x$ such that $y\in k$. In particular, $\u x \subseteq \u y\subseteq k$ and therefore
	\[
	\ros{F}{\u x}{ k} \circ \mu = \ros{F}{\u y}{ k} \circ \ros{F}{\u x}{ \u y} \circ \mu = \ros{F}{\u y}{ k} \circ \phi_y,
	\]
	which coincides with $\psi_k$. To see that $\mu$ is unique with this property, it suffices to note that any $\nu$ satisfying $\psi_k = \ros{F}{\u x}{ k} \circ \nu$ for all $k\in\Scott(L)$ such that $x\in k$ must satisfy $\phi_y = \ros{F}{\u x}{ \u y} \circ \nu$ for all $y\ll x$. Just recall that, as pointed out above, whenever $y\ll x$ there is $k\in\Scott(L)$ satisfying $\u x \subseteq k\subseteq \u y$. Hence,
	\[
	\ros{F}{\u x}{ \u y} \circ \nu = \ros{F}{k}{ \u y} \circ \ros{F}{\u x}{ k} \circ \nu = \ros{F}{k}{ \u y} \circ \psi_k = \phi_y. \qedhere
	\]
\end{proof}

Recall that $L$ is an arbitrary domain and $\cat{C}$ is a bicomplete category. The next proposition provides a description of the objects that are fixed by the adjunction in \cref{e:adjunction}.

\begin{proposition} \label{p:directed-sheaves}
	The adjoint pair of functors
	\[
		\begin{tikzcd}[column sep = 4em]
			{[\Scott(L), \cat{C}]} \arrow[yshift=4pt]{r}{\lambda^*\Ran_{\kappa}} \arrow[r,phantom,"\text{\tiny{$\top$}}"] & {[L^\op, \cat{C}]} \arrow[yshift=-4pt]{l}{\kappa^*\Lan_{\lambda}}
		\end{tikzcd}
	\]
	restricts to an equivalence $\omega\text{-}\colim[\Scott(L), \cat{C}] \simeq \omega\text{-}\lim[L^\op, \cat{C}]$.
\end{proposition}

\begin{proof}
	Denote by $\cat{F}_1$ the full subcategory of $[\mathcal{F}, \cat{C}]$ defined by those functors $F$ such that $\lambda^* F\in \omega\text{-}\lim[L^\op, \cat{C}]$ and $\epsilon_F$ is an isomorphism. We claim that the adjunction 
	\[
	\Lan_{\lambda}\dashv \lambda^*\colon [\mathcal{F}, \cat{C}] \to [L^\op,\cat{C}]
	\] 
	restricts to an equivalence between $\cat{F}_1$ and $\omega\text{-}\lim[L^\op, \cat{C}]$. 
	
	By definition of $\cat{F}_1$, $\lambda^*$ restricts to a functor $\cat{F}_1 \to \omega\text{-}\lim[L^\op, \cat{C}]$. To see that $\Lan_{\lambda}$ restricts to a functor ${\omega\text{-}\lim[L^\op, \cat{C}]\to \cat{F}_1}$, consider an arbitrary $G\in \omega\text{-}\lim[L^\op, \cat{C}]$. Then $\eta_G\colon G\cong \lambda^*\Lan_{\lambda} G$ by \cref{second-adj} in \cref{l:fixed-side-Kan}, and so we see that $\lambda^*\Lan_{\lambda} G$ belongs to $\omega\text{-}\lim[L^\op, \cat{C}]$. Furthermore, using again the fact that $\eta$ is a natural isomorphism, it follows from the triangle identities for adjunctions that $\epsilon \Lan_{\lambda}$ is a natural isomorphism. In particular, $\epsilon_{\Lan_{\lambda} G}$ is an isomorphism. Hence, $\Lan_{\lambda} G \in \cat{F}_1$.
	
	This shows that the adjunction $\Lan_{\lambda}\dashv \lambda^*$ restricts to an adjunction between $\cat{F}_1$ and $\omega\text{-}\lim[L^\op, \cat{C}]$. In turn, the latter adjunction is an equivalence because all objects of $\cat{F}_1$ are fixed by definition, and all objects of $\omega\text{-}\lim[L^\op, \cat{C}]$ are fixed by \cref{second-adj} in \cref{l:fixed-side-Kan}.
	
	Now, let $\cat{F}_2$ be the full subcategory of $[\mathcal{F}, \cat{C}]$ defined by those functors $F$ such that $\kappa^* F\in \omega\text{-}\colim[\Scott(L), \cat{C}]$ and $\tilde{\eta}_F$ is an isomorphism. By similar reasoning to the one above, the adjunction $\kappa^*\dashv \Ran_{\kappa}\colon [\Scott(L),\cat{C}]\to [\mathcal{F}, \cat{C}]$ restricts to an equivalence between $\omega\text{-}\colim[\Scott(L), \cat{C}]$ and $\cat{F}_2$. 
	
	Finally, the statement follows by noting that $\cat{F}_1=\cat{F}_2$ by \cref{l:directed-sheaves-F}.
\end{proof}

Suppose for a moment that $\cat{C}$ is a bicomplete regular category and the domain $L$ is a stably continuous lattice. Then $\Scott(L)$ is also a stably continuous lattice; in particular, a complete lattice. So, we can consider the category $\KSh(\Scott(L)^\op,\cat{C})$ of $\cat{C}$-valued $\K$-sheaves over $\Scott(L)^\op$. This is a full subcategory of $[\Scott(L),\cat{C}]$ and we have a composite functor
\begin{equation*}
	\begin{tikzcd}[column sep = 4em]
	{\KSh(\Scott(L)^\op,\cat{C})} \arrow[hookrightarrow]{r} & {[\Scott(L),\cat{C}]} \arrow{r}{\lambda^*\circ \Ran_{\kappa}} & {[L^\op,\cat{C}]}.
	\end{tikzcd}
\end{equation*}

In order to characterise the image of this functor, we introduce the notion of \emph{$\Omega$-sheaf}.
Intuitively, whereas $\ksat$-sheaves on a (stably compact) space $X$ are defined on the lattice $\K(X)$ of compact saturated subsets of $X$, $\Omega$-sheaves are defined on the frame $\Omega(X)$ of opens of $X$---hence the nomenclature. The relation between $\Omega$-sheaves and ordinary sheaves is explained in \cref{p:sheaf-om-sheaf} and \cref{rem:Omega-sheaves-vs-sheaves-caveat}.

\begin{definition}\label{d:Om-sheaf}
Let $\cat{D}$ be a category and let $P$ be a complete lattice.
	A \emph{$\cat{D}$-valued $\Omega$-sheaf} on $P$ is a functor $F \colon P^\op \to \cat{D}$ that satisfies the following properties:
	\begin{enumerate}[label=\textnormal{($\Omega$\arabic*)}]
		\item\label{Om-empty} $F(\bot)$ is a subterminal object of $\cat{D}$, i.e.\ the unique arrow $F(\bot)\to\1$ is monic.
		\item\label{Om-pullback} For all $x, y \in P$, the following is a pullback square in $\cat{D}$:
		\[\begin{tikzcd}[column sep = 4em]
			F(x \vee y) \arrow{r}{\ros{F}{x \vee y}{ x}} \arrow[swap]{d}{\ros{F}{x \vee y}{ y}} \arrow[dr, phantom, "\lrcorner", very near start, xshift=-6pt] & F(x) \arrow{d}{\ros{F}{x}{ x \wedge y}}\\
			F(y) \arrow{r}{\ros{F}{y}{ x \wedge y}} & F(x \wedge y)
		\end{tikzcd}\]
		\item \label{Om-bigvee} $F$ preserves codirected limits. I.e., for all directed subsets $\D\subseteq P$, the~cone $(\ros{F}{\sup \D}{p} \colon F(\sup \D) \to F(p))_{p \in \D}$ is a limit of the restriction of $F$ to~$\D$.
	\end{enumerate}
	We denote by $\OSh(P,\cat{D})$ the full subcategory of $[P^\op, \cat{D}]$ defined by the $\Omega$-sheaves. 
\end{definition}

\begin{remark}
	Note that conditions~\ref{Om-empty} and~\ref{Om-pullback} coincide, respectively, with conditions~\ref{K-empty} and~\ref{K-pullback} in the definition of $\K$-sheaf. By contrast,~\ref{Om-bigvee} is \emph{dual} to~\ref{K-codirected-meet}.
\end{remark}

\begin{remark}
	By \cref{l:Om3-K3}, if $P$ is a continuous lattice and $F \colon P^\op \to \cat{C}$ is a functor, then condition \ref{Om-bigvee} can be equivalently stated as follows: For all $x \in P$, the cone $(\ros{F}{x}{y}\colon F(x) \to F(y))_{y \in \dd x}$ is a (codirected) limit of the restriction of $F$ to $\dd x$.
\end{remark}

\begin{lemma} \label{l:rewrite-2-3}
	Suppose that $\cat{D}$ is a category and $P$ is a complete lattice. The following statements are equivalent for all functors $F\colon P^\op\to\cat{D}$:
	\begin{enumerate}
		\item \label{i:om2-3} $F$ satisfies \ref{Om-pullback} and \ref{Om-bigvee}.
		\item \label{i:limit} For every non-empty subset $S\subseteq P$ that is closed under binary meets, the cone $(\ros{F}{\sup S}{y} \colon F(\sup S) \to F(y))_{y \in S}$ is a limit of the restriction of $F$ to $S$.
	\end{enumerate}
\end{lemma}

\begin{proof}
	Suppose \ref{i:om2-3} holds and let $S\subseteq P$ be a non-empty subset closed under binary meets. The set $T\subseteq P$ obtained by closing $S$ under binary joins is directed and so, by \ref{Om-bigvee}, the cone $(\ros{F}{\sup T}{y} \colon F(\sup T) \to F(y))_{y \in T}$ is a limit of the restriction of $F$ to~$T$. It follows from \ref{Om-pullback} that the cone $(\ros{F}{\sup S}{y} \colon F(\sup S) \to F(y))_{y \in S}$ is a limit of the restriction of $F$ to $S$. Just observe that $\sup S = \sup T$ and condition \ref{Om-pullback} allows to extend any compatible cone over the restriction of $F$ to $S$ to a unique compatible cone over the restriction of $F$ to $T$.
	
	Conversely, assume that \ref{i:limit} holds. Then \ref{Om-pullback} follows by setting $S \coloneqq \{x,y, x \land y\}$. 
	With regards to \ref{Om-bigvee}, let $\D\subseteq P$ be a directed subset and let $E$ be the closure of $\D$ under binary meets.
	As $\D$ is non-empty, so is $E$. Also, since $\D$ is a cofinal subset of $E$, we have $\sup{\D} = \sup{E}$.
	By \cref{i:limit}, the cone $(\ros{F}{\sup \D}{y} \colon F(\sup \D) \to F(y))_{y \in E}$ is a limit of the restriction of $F$ to $E$.
	Moreover, a straightforward argument that uses the fact that $\D$ is a directed cofinal subset of $E$ shows that $(\ros{F}{\sup \D}{y} \colon F(\sup \D) \to F(y))_{y \in \D}$ is a limit of the restriction of $F$ to $\D$.
\end{proof}

Recall that a sheaf of sets on a topological space $X$ can be characterised as a presheaf $F \colon \Omega(X)^\op \to \Set$ such that, for every set of opens $S\subseteq \Omega(X)$ that is closed under binary intersections, the cone \[(\ros{F}{\bigcup S}{U} \colon F\big(\bigcup S\big) \to F(U))_{U \in S}\] is a limit of the restriction of $F$ to~$S$. For $S = \emptyset$, this amounts to saying that $F(\emptyset)$ is a terminal object of $\Set$, i.e.\ a one-element set, which is a strengthening of \ref{Om-empty}. More generally, for any frame $M$, a sheaf of sets over $M$ is a presheaf $F\colon M^\op \to\Set$ such that, for every set $S\subseteq M$ closed under binary meets, the cone \[(\ros{F}{\sup S}{a} \colon F\big(\sup S\big) \to F(a))_{a \in S}\] is a limit of the restriction of $F$ to~$S$.

The previous description of sheaves as ``limit-preserving presheaves'' remains valid when the category of sets is replaced with any variety of (finitary) algebras, and has been exploited to propose a notion of sheaf with values in an arbitrary category $\cat{D}$---simply by replacing $\Set$ with $\cat{D}$---see e.g.\ \cite{Gray1965}. Adopting this notion of $\cat{D}$-valued sheaf on a frame, we have the following immediate consequence of \cref{l:rewrite-2-3}:

\begin{proposition} \label{p:sheaf-om-sheaf}
	Let $\cat{D}$ be a category and let $M$ be a frame.
	The following statements are equivalent for all presheaves $F \colon M^\op \to \cat{D}$:
	\begin{enumerate}		
		\item $F$ is a $\cat{D}$-valued sheaf.
		\item $F$ is a $\cat{D}$-valued $\Omega$-sheaf such that $F(\emptyset)$ is a terminal object of $\cat{D}$.
	\end{enumerate}
\end{proposition}

Thus the notion of $\Omega$-sheaf of sets over a frame (and, in particular, over a topological space) coincides with the classical notion of sheaf of sets, with the only exception that we allow the empty presheaf. This is the content of the following remark:
\begin{remark}\label{rem:Omega-sheaves-vs-sheaves-caveat}
Let $M$ be a frame.
When $\cat{D} = \Set$, the only proper subterminal object is the empty set. Thus, any $\Omega$-sheaf $F \colon M^\op \to \Set$ that is not a sheaf satisfies $F(\bot) = \emptyset$. But for every $a\in M$ there is a restriction function $\ros{F}{a}{\bot}\colon F(a) \to F(\bot) = \emptyset$, which implies that $F(a) = \emptyset$. Thus, a presheaf of sets $M^\op \to \Set$ is an $\Omega$-sheaf if, and only if, it is either a sheaf or the initial presheaf (i.e., the constant presheaf of value~$\emptyset$). 

A similar reasoning applies when $\cat{D}$ is any variety of algebras whose signature contains no constant symbols, e.g.\ the variety of semigroups. On the other hand, if $\cat{D}$ is a variety of algebras whose signature contains at least one constant symbol, then it admits no proper subterminal objects (just observe that any algebra in $\cat{D}$ has at least one element). In this case, the notions of sheaf and $\Omega$-sheaf coincide for any presheaf $M^\op \to \cat{D}$.
\end{remark}

We will now proceed to compare the notions of $\K$-sheaf and $\Omega$-sheaf. To this end, we shall suppose from now on that $\cat{C}$ is a bicomplete regular category and $L$ is a stably continuous lattice.
\begin{definition}\label{def:A-Kappa-A-Omega}
	We consider the following full subcategories $\cat{A}_{\mathcal{K}}$ and $\cat{A}_{\Omega}$ of $[\F, \cat{C}]$:
	\begin{itemize}[leftmargin=*]
		\item\label{A-Kappa} $\cat{A}_{\mathcal{K}}$ consists of those $F$ such that $\kappa^*F$ is a $\K$-sheaf and $\tilde{\eta}_F$ is an isomorphism. 
		\item\label{A-Omega} $\cat{A}_{\Omega}$ consists of those $F$ such that $\lambda^*F$ is an $\Omega$-sheaf and $\epsilon_F$ is an isomorphism. 
		\end{itemize}
\end{definition}

The next lemma follows by reasoning as in the proof of \cref{p:directed-sheaves}:
\begin{lemma}\label{l:equiv-fixed}
The following statements hold:
\begin{enumerate}[label=(\alph*)]
\item\label{eq:A_K} The adjunction $\kappa^*\dashv \Ran_{\kappa}$ restricts to an equivalence ${\KSh(\Scott(L)^\op,\cat{C})\simeq \cat{A}_{\mathcal{K}}}$.
\item\label{eq:A_Omega} The adjunction $\Lan_{\lambda}\dashv \lambda^*$ restricts to an equivalence $\OSh(L,\cat{C})\simeq \cat{A}_{\Omega}$.
\end{enumerate}
\end{lemma}

\begin{lemma}\label{l:AK-included-in-AOm}
$\cat{A}_{\mathcal{K}}$ is a (full) subcategory of $\cat{A}_{\Omega}$.
\end{lemma}
\begin{proof}
Fix an arbitrary functor $F\colon \F\to\cat{C}$ that belongs to $\cat{A}_{\mathcal{K}}$. We must prove that $\lambda^*F$ is an $\Omega$-sheaf and $\epsilon_F$ is an isomorphism. Since $\tilde{\eta}_F\colon F\to \Ran_{\kappa}\kappa^* F$ is an isomorphism, the formula for pointwise right Kan extensions implies that, for all $x\in L$, 
\[
\lambda^*F(x) \cong \lim_{x\in k} \kappa^* F(k)
\]
where $k\in\Scott(L)$. Just observe that $x\in k$ if and only if $\u x\subseteq k$ in $\F$. 

Condition \ref{Om-empty} in \cref{d:Om-sheaf} is trivially satisfied, since the unique (Scott-open) filter containing the bottom element of $L$ is the improper one, which is the bottom element of $\Scott(L)^\op$. Hence $\lambda^*F(\bot)\cong \kappa^* F(L)$, which is subterminal because $\kappa^* F$ satisfies \ref{K-empty}. 
For \ref{Om-pullback}, for all $x,y\in L$ we have
\begin{align*}
\lambda^*F(x) \times_{\lambda^*F(x\wedge y)} \lambda^*F(y) & \cong \lim_{x\in k}{\kappa^* F(k)} \times_{\lim{\kappa^* F(k\vee k')}} \lim_{y\in k'}{\kappa^* F(k')} \\
& \cong \lim_{x\in k, \, y\in k'}{(\kappa^* F(k) \times_{\kappa^* F(k\vee k')} \kappa^* F(k'))} \\
& \cong \lim_{x\in k, \, y\in k'}{\kappa^*F(k\wedge k')} \tag*{$\kappa^* F$ satisfies \ref{K-pullback}}\\
& \cong \lim_{x\vee y \in l}{\kappa^* F(l)} \tag*{\cref{l:wilker} and coinitiality\footnotemark}\\
& \cong \lambda^*F(x\vee y)
\end{align*}
where in the first step we used the fact that $\{l\mid x\wedge y\in l\} = \{k\vee k' \mid x\in k, \, y\in k'\}$ and so 
\[
\lim_{x\wedge y \in l}{\kappa^* F(l)} \cong \lim_{x\in k, \, y\in k'}{\kappa^* F(k\vee k')}.
\]
Finally, since $\kappa^* F$ preserves directed colimits (by definition of $\K$-sheaf) and $\tilde{\eta}_F$ is an isomorphism, \cref{l:directed-sheaves-F} entails that $\lambda^* F$ satisfies \ref{Om-bigvee} and $\epsilon_F$ is an isomorphism.
\footnotetext{By \emph{coinitiality} we understand the order-theoretic dual to the notion of cofinality.}
\end{proof}

We record the following immediate consequence of \cref{l:equiv-fixed,l:AK-included-in-AOm}. Recall that $L$ denotes an arbitrary stably continuous lattice, and $\cat{C}$ a bicomplete regular category.
\begin{proposition}\label{pr:inclusion-K-into-Omega}
There is a fully faithful functor 
\begin{equation}\label{eq:composite-ff}
\KSh(\Scott(L)^\op,\cat{C}) \to \OSh(L,\cat{C})
\end{equation}
given by the following composition:
\begin{equation*}
\KSh(\Scott(L)^\op,\cat{C}) \simeq \cat{A}_{\mathcal{K}} \hookrightarrow \cat{A}_{\Omega} \simeq \OSh(L,\cat{C}).
\end{equation*}
\end{proposition}
Next, we shall see that the fully faithful functor in the previous proposition is an equivalence of categories whenever directed colimits in $\cat{C}$ commute with finite limits. 

\begin{remark}
Directed colimits commute with finite limits in $\Set$ (more generally, in any Grothendieck topos), as well as in algebraic categories (i.e., categories of models of Lawvere theories, or equivalently varieties of finitary algebras); see e.g.\ \cite[Corollary~3.4.3]{Borceux2}. 

Furthermore, suppose $\cat{D}$ is a Barr-exact category with a regular generator that admits all small copowers. Then, by a result of Vitale~\cite{Vitale1998}, directed colimits in $\cat{D}$ exist and commute with finite limits if, and only if, $\cat{D}$ is equivalent to the \emph{localization} (i.e., a reflective subcategory such that the reflector preserves finite limits) of an algebraic category.
\end{remark}

The following theorem provides an equivalence between $\K$-sheaves and $\Omega$-sheaves, and is akin to a result of Lurie for sheaves on locally compact Hausdorff spaces with values in $\infty$-categories \cite[Corollary~7.3.4.10]{Lurie2009}. The two results are incomparable: we work with ordinary categories but consider, more generally, sheaves on stably continuous lattices. 
\begin{theorem}\label{equiv-Om-K-sheaves}
If directed colimits in $\cat{C}$ commute with finite limits, there is an equivalence of categories $\KSh(\Scott(L)^\op,\cat{C}) \simeq \OSh(L,\cat{C})$.
\end{theorem}
\begin{proof}
Consider the fully faithful functor $\KSh(\Scott(L)^\op,\cat{C}) \to \OSh(L,\cat{C})$ in \cref{eq:composite-ff}. In view of the definition of the latter, it suffices to show that~$\cat{A}_{\Omega}$ is a (full) subcategory of~$\cat{A}_{\mathcal{K}}$, for then $\cat{A}_{\mathcal{K}} = \cat{A}_{\Omega}$. To this end, fix an arbitrary functor $F\colon \F \to \cat{C}$ in~$\cat{A}_{\Omega}$. We must show that $\kappa^*F$ is a $\K$-sheaf and $\tilde{\eta}_F$ is an isomorphism.
Since the component of the counit $\epsilon_F\colon \Lan_{\lambda}\lambda^* F\to F$ is an isomorphism, for all $k\in\Scott(L)$ we have
\[
\kappa^* F(k) \cong \colim_{x\in k} \lambda^* F(x)
\]
by the formula for pointwise left Kan extensions. Note that the colimit above is directed because $k$ is codirected and $\lambda^* F$ is contravariant. 

Clearly, $\kappa^* F$ satisfies condition \ref{K-empty} in \cref{d:K-sheaf-on-complete-lattice}. 
Just observe that the bottom element of $\Scott(L)^\op$ is the improper filter $L$, and so every arrow in the colimit cocone 
\[
\{\lambda^* F(x) \to \kappa^* F(L)\mid x\in L\}
\]
factors through $\lambda^* F(\bot) \to \kappa^* F(L)$. Hence $\kappa^* F(L) \cong \lambda^* F (\bot)$, which is subterminal because $\lambda^* F$ satisfies \ref{Om-empty}.
For \ref{K-pullback}, using the fact that directed colimits in $\cat{C}$ commute with finite limits, for all $k,l\in \Scott(L)$ we have
\begin{align*}
\kappa^* F(k) \times_{\kappa^* F(k\vee l)} \kappa^* F(l) & \cong \colim_{x\in k}{\lambda^* F(x)} \times_{\colim{\lambda^* F(x\wedge y)}} \colim_{y\in l}{F(y)} \\
& \cong \colim_{x\in k, \, y\in l}{(\lambda^* F(x) \times_{\lambda^* F(x\wedge y)} \lambda^* F(y))} \\
& \cong \colim_{x\in k, \, y\in l}{\lambda^* F(x\vee y)}  \tag*{$\lambda^* F$ satisfies \ref{Om-pullback}} \\
& \cong \colim_{z\in k\wedge l}{\lambda^* F(z)} \\ 
& \cong \kappa^* F(k\wedge l)
\end{align*}
where in the first step we used the fact that the supremum of the Scott-open filters $k$ and $l$ is $\u \{x\wedge y\mid x\in k, \, y\in l\}$ and thus, by coinitiality,
\[
\colim_{z\in k\vee l}{\lambda^* F(z)} \cong \colim_{x\in k, \, y\in l}{\lambda^* F(x\wedge y)}.
\]
Similarly, the penultimate step holds because $k\wedge l = \{x\vee y\mid x\in k, \, y\in l\}$.
Moreover, because $\lambda^* F$ preserves codirected limits (by definition of $\Omega$-sheaf) and $\epsilon_F$ is an isomorphism, \cref{l:directed-sheaves-F} shows that $\kappa^* F$ satisfies \ref{K-codirected-meet} and $\tilde{\eta}_F$ is an isomorphism.
\end{proof}

The equivalence of categories in the previous theorem induces an equivalence of sheaf representations in the following sense. 

\begin{definition}
For every object $A$ of a category $\cat{D}$, an \emph{$\Omega$-sheaf representation} of $A$ over a complete lattice $P$ is a pair $(F,\phi)$ where $F\colon P^\op \to\cat{D}$ is an $\Omega$-sheaf and $\phi\colon A \to F(\top)$ an isomorphism in $\cat{D}$. 

We denote by $\OmShr{A}(P,\cat{D})$ the category of $\Omega$-sheaf representations of $A$ over $P$; a morphism $(F,\phi)\to (G,\psi)$ in this category is a natural transformation $\alpha\colon F\Rightarrow G$ such that $\alpha_{\top} \circ \phi = \psi$. 
\end{definition}

\begin{theorem}\label{thm:omega-k-representation}
If directed colimits in $\cat{C}$ commute with finite limits, then for any $A\in \cat{C}$ there is an equivalence of categories $\KShr{A}(\Scott(L)^\op,\cat{C}) \simeq \OmShr{A}(L,\cat{C})$.
\end{theorem}
\begin{proof}
Fix an arbitrary object $A$ of $\cat{C}$ and an $\Omega$-sheaf representation $(F,\phi)$ of $A$ over $L$. The top element of $\Scott(L)^\op$ is the filter $\{\top\}$, where $\top$ is the top element of $L$ (note that $\{\top\}$ is Scott-open because $\top \ll \top$ in a stably continuous lattice). We have
\[
\lambda^* \Lan_{\lambda} F (\top) = \Lan_{\lambda} F (\{\top\}) = \kappa^* \Lan_{\lambda} F (\{\top\})
\]
and so, by \cref{second-adj} in \cref{l:fixed-side-Kan}, we obtain an isomorphism $(\eta_F)_{\top}$ from $F(\top)$ to $\kappa^* \Lan_{\lambda} F (\{\top\})$.
The composite arrow
\[\begin{tikzcd}
\phi^*\colon A \arrow{r}{\phi} & F(\top) \arrow{r}{(\eta_F)_{\top}} & \kappa^* \Lan_{\lambda} F (\{\top\})
\end{tikzcd}\] 
is then an isomorphism and the pair $(\kappa^* \Lan_{\lambda} F, \phi^*)$ is a $\K$-sheaf representation of $A$ over $\Scott(L)^\op$. Further, if $\alpha\colon(F,\phi)\to (G,\psi)$ is a morphism in $\OmShr{A}(P,\cat{C})$ then 
\[
\kappa^* \Lan_{\lambda}(\alpha)\colon (\kappa^* \Lan_{\lambda} F, \phi^*)\to (\kappa^* \Lan_{\lambda} G, \psi^*)
\] 
is a morphism in $\KShr{A}(\Scott(L)^\op,\cat{C})$. Just observe that the square in the following diagram commutes by naturality of $\eta$,
\[\begin{tikzcd}[row sep = 1em, column sep = 3em]
{} & F(\top) \arrow{r}{(\eta_F)_{\top}} \arrow{dd}[swap]{\alpha_{\top}} & \kappa^* \Lan_{\lambda} F (\{\top\}) \arrow{dd}{(\kappa^* \Lan_{\lambda}(\alpha))_{\{\top\}}} \\
A \arrow{ur}{\phi} \arrow{dr}[swap]{\psi} & & \\
{} & G(\top) \arrow{r}{(\eta_G)_{\top}} & \kappa^* \Lan_{\lambda} G(\{\top\})
\end{tikzcd}\]
and so $(\kappa^* \Lan_{\lambda}(\alpha))_{\{\top\}} \circ \phi^* = \psi^*$.

This yields a functor 
\begin{equation}\label{eq:equiv-sheaf-repr}
\OmShr{A}(L,\cat{C})\to \KShr{A}(\Scott(L)^\op,\cat{C}).
\end{equation}
We claim that the latter is an equivalence of categories.
It follows from (the proof of) \cref{equiv-Om-K-sheaves} that $\kappa^* \Lan_{\lambda}\colon \OSh(L,\cat{C}) \to \KSh(\Scott(L)^\op,\cat{C})$ is an equivalence. 
Hence the functor in \cref{eq:equiv-sheaf-repr} is faithful. To see that it is full, fix an arbitrary morphism 
\[
\beta\colon (\kappa^* \Lan_{\lambda} F, \phi^*)\to (\kappa^* \Lan_{\lambda} G, \psi^*)
\] 
in $\KShr{A}(\Scott(L)^\op,\cat{C})$. Since $\kappa^* \Lan_{\lambda}$ is full, there is $\alpha\colon F\Rightarrow G$ such that ${\kappa^* \Lan_{\lambda}(\alpha)= \beta}$. 
Note that
\begin{align*}
(\eta_G)_{\top} \circ \alpha_{\top} \circ \phi &= (\kappa^* \Lan_{\lambda}(\alpha))_{\{\top\}} \circ (\eta_F)_{\top} \circ \phi \tag*{Naturality of $\eta$} \\
&= \beta_{\{\top\}} \circ \phi^* \\
&= \psi^* \\
&= (\eta_G)_{\top} \circ \psi
\end{align*}
and so $ \alpha_{\top} \circ \phi = \psi$ because $(\eta_G)_{\top}$ is an isomorphism. It follows that $\alpha$ is a morphism of $\Omega$-sheaf representations and thus the functor in \cref{eq:equiv-sheaf-repr} is full. Finally, fix an arbitrary object $(G,\psi)\in \KShr{A}(\Scott(L)^\op,\cat{C})$. Because $\kappa^* \Lan_{\lambda}$ is essentially surjective, there exist an $\Omega$-sheaf $F\in \OSh(L,\cat{C})$ and a natural isomorphism $\delta\colon \kappa^* \Lan_{\lambda} F \Rightarrow G$. Let 
\[
\phi\coloneqq (\delta_{\{\top\}}\circ (\eta_F)_{\top})^{-1}\circ\psi\colon A \to F(\top).
\] 
Then $(F,\phi)$ is an object of $\OmShr{A}(L,\cat{C})$ whose image under the functor in \cref{eq:equiv-sheaf-repr} is isomorphic to $(G,\psi)$. We conclude that the latter functor is an equivalence of categories.
\end{proof}

The classical notion of soft sheaf can be extended to arbitrary $\Omega$-sheaves:
\begin{definition}\label{def:soft-Omega-sheaf}
Let $P$ be a complete lattice and let $\cat{D}$ be a category admitting directed colimits. An $\Omega$-sheaf $F\colon P^\op\to\cat{D}$ is \emph{soft} if, for all Scott-open filters $k\in\Scott(P)$, the canonical colimit arrow 
\[
F(\top) \to \colim_{x\in k}{F(x)}
\]
is a regular epimorphism.

An $\Omega$-sheaf representation $(F,\phi)$ is said to be \emph{soft} if $F$ is a soft $\Omega$-sheaf. The full subcategory of $\OmShr{A}(P,\cat{D})$ defined by the soft $\Omega$-sheaf representations is denoted by 
\[
\sOmShr{A}(P,\cat{D}).
\]
\end{definition}

\begin{remark}
	In the case where $P=\Omega(X)$ for a sober space $X$, and $\cat{D}$ is $\Set$---or, more generally, any variety of finitary algebras---the Hofmann--Mislove theorem implies that the notion of softness introduced in \cref{def:soft-Omega-sheaf} coincides with the ordinary one (always with the \emph{caveat} outlined in \cref{rem:Omega-sheaves-vs-sheaves-caveat}). Cf.\ e.g.\ \cite[Remark~3.3]{GehrkeGool2018}.
\end{remark}

\begin{remark}\label{rem:soft-sh-repr-stalks-as-quotients}
Let $\cat{D}$ be $\Set$ or, more generally, any variety of algebras. Fix an object $A\in \cat{D}$ and assume that $F\colon \Omega(X)^\op\to\cat{D}$ is a soft $\Omega$-sheaf representation of $A$ over a sober space $X$. Reasoning as in \cref{rem:K-sheaf-encodes-stalks}, we see that for all $x\in X$ the canonical arrow
\[
A \to \colim_{x\in U}{F(U)}
\]
whose codomain is the stalk of $F$ at $x$ is a regular epimorphism. This exhibits each stalk of $F$ as a quotient of $A$.
\end{remark}

\begin{lemma}\label{lem:soft-soft}
Suppose that directed colimits in $\cat{C}$ commute with finite limits.
The following statements are equivalent for all functors $F\colon L^\op \to\cat{C}$:
\begin{enumerate}
\item $F$ is a soft $\Omega$-sheaf.
\item $\kappa^*\Lan_{\lambda} F$ is a soft $\K$-sheaf.
\end{enumerate}
\end{lemma}
\begin{proof}
Fix an arbitrary functor $F\colon L^\op \to\cat{C}$. In view of \cref{thm:omega-k-representation}, $F$ is an $\Omega$-sheaf over $L$ precisely when $\kappa^*\Lan_{\lambda} F$ is a $\K$-sheaf over $\Scott(L)^\op$. Thus, it suffices to show that $F$ is soft (as an $\Omega$-sheaf) if, and only if, $\kappa^*\Lan_{\lambda} F$ is soft (as a $\K$-sheaf). 

Recall that the top element of $\Scott(L)^\op$ is $\{\top\}$, where $\top$ is the top element of $L$. Hence, $\kappa^*\Lan_{\lambda} F$ is soft precisely when, for all $k\in\Scott(L)$, the arrow 
\[
\ros{\kappa^*\Lan_{\lambda} F}{\{\top\}}{k}\colon (\kappa^*\Lan_{\lambda} F)(\{\top\}) \to (\kappa^*\Lan_{\lambda} F)(k)
\]
is a regular epimorphism. Moreover, recall from \cref{r:explicitly} that, for all $l\in\Scott(L)$, $(\kappa^*\Lan_{\lambda} F)(l) \cong \colim_{x\in l}{F(x)}$. Under this isomorphism, $\ros{\kappa^*\Lan_{\lambda} F}{\{\top\}}{k}$ can be identified with the canonical colimit arrow
\[
F(\top) \to \colim_{x\in k}{F(x)}.
\]
Therefore, $\kappa^*\Lan_{\lambda} F$ is soft if, and only if, so is $F$.
\end{proof}

\begin{remark}\label{rem:one-dir-softness-preserved}
Even if direct colimits in $\cat{C}$ fail to commute with finite limits, by \cref{pr:inclusion-K-into-Omega} there is a full and faithful functor $\KSh(\Scott(L)^\op,\cat{C}) \hookrightarrow \OSh(L,\cat{C})$. The second part of the proof of \cref{lem:soft-soft}, combined with \cref{p:directed-sheaves}, then shows that this functor sends soft $\K$-sheaves to soft $\Omega$-sheaves.
\end{remark}

Combining the previous observations, we obtain an equivalence between soft $\K$-sheaf representations and soft $\Omega$-sheaf representations:
\begin{proposition}\label{pr:soft-sheaf-repr-equiv}
If directed colimits in $\cat{C}$ commute with finite limits, then for any $A\in \cat{C}$ there is an equivalence of categories $\sKShr{A}(\Scott(L)^\op,\cat{C}) \simeq \sOmShr{A}(L,\cat{C})$.
\end{proposition}
\begin{proof}
By (the proof of) \cref{thm:omega-k-representation}, combined with \cref{lem:soft-soft}.
\end{proof}

\begin{remark}
Under the assumptions of \cref{pr:soft-sheaf-repr-equiv}, the category $\sOmShr{A}(L,\cat{C})$ of soft $\Omega$-sheaf representations of $A$ over $L$ is a (large) preorder. This follows from \cref{pr:soft-sheaf-repr-equiv}, recalling that $\sKShr{A}(\Scott(L)^\op,\cat{C})$ is a preorder by \cref{soft-thin} in \cref{l:soft-repr-thin}.
\end{remark}

We can finally state our main result, which characterises soft $\Omega$-sheaf representations for a broad class of regular categories. Recall that $L$ denotes an arbitrary stably continuous lattice, and $\cat{C}$ a bicomplete regular category.
\begin{theorem}\label{t:gen}
Suppose that directed colimits in $\cat{C}$ commute with finite limits, and let $A\in\cat{C}$. 
Let $\cat{M}$ be the (large) sub-preorder of $[\Scott(L),\RegEpi{A}]$ consisting of those maps that preserve finite infima and arbitrary suprema, and whose images consist of pairwise ker-commuting elements. There is an equivalence of categories
\[
\cat{M} \simeq \sOmShr{A}(L,\cat{C}).
\]
\end{theorem}
\begin{proof}
By \cref{th:equiv-soft-sheaf-repr} and \cref{pr:soft-sheaf-repr-equiv}.
\end{proof}

In the same vein of \cref{cor:MAIN}, we state a consequence of \cref{t:gen} obtained by taking the poset reflections of the categories involved, assuming that $\cat{C}$ is Barr-exact. Recall that $\llbracket \sOmShr{A}(L,\cat{C}) \rrbracket$ denotes the poset reflection of $\sOmShr{A}(L,\cat{C})$; the objects of $\llbracket \sOmShr{A}(L,\cat{C}) \rrbracket$ are isomorphism classes of soft $\Omega$-sheaf representations of $A$ over $L$.

\begin{corollary}\label{cor:pos-refl-omega} 
	Suppose that $\cat{C}$ is Barr-exact and directed colimits in $\cat{C}$ commute with finite limits, and let $A\in\cat{C}$. 
	Let $\cat{N}$ be the (large) sub-poset of $[\Scott(L),\Equiv{A}]$ consisting of those maps that preserve finite infima and arbitrary suprema, and whose images consist of pairwise commuting equivalence relations. There is an order isomorphism 
\[
\cat{N}\cong\llbracket \sOmShr{A}(L,\cat{C}) \rrbracket. 
\]
\end{corollary}
\begin{proof}
Note that two (large) posets that are equivalent as categories must be order isomorphic. Thus, by \cref{t:gen}, there is an order isomorphism $\llbracket \cat{M} \rrbracket \cong \llbracket \sOmShr{A}(L,\cat{C}) \rrbracket$. In turn, as pointed out in the proof of \cref{cor:MAIN}, $\llbracket \cat{M} \rrbracket \cong \cat{N}$.
\end{proof}

\Cref{t:gen,cor:pos-refl-omega} are generalisations of \cite[Theorem~3.10]{GehrkeGool2018} from the framework of varieties of finitary algebras to that of regular and Barr-exact categories, respectively. We end this section with some remarks concerning the previous results.
	
\begin{remark}\label{rem:well-powered-sheaf-rep}
As with \cref{cor:MAIN}, if $\cat{C}$ is a well-powered category then all posets in the statement of \cref{cor:pos-refl-omega} are small (cf.\ \cref{rem:well-powered}).
\end{remark}

\begin{remark}\label{rem:M-into-Omega-shvs}
Even when directed colimits in $\cat{C}$ do not commute with finite limits, and so \cref{t:gen} does not apply, by \cref{pr:inclusion-K-into-Omega} there is a full and faithful functor 
\[
\KSh(\Scott(L)^\op,\cat{C}) \hookrightarrow \OSh(L,\cat{C}).
\]
The latter induces, for every object $A$ of $\cat{C}$, a full and faithful functor 
\[
\sKShr{A}(\Scott(L)^\op,\cat{C}) \hookrightarrow \sOmShr{A}(L,\cat{C}). 
\]
Cf.~\cref{rem:one-dir-softness-preserved}.
Hence, in view of \cref{th:equiv-soft-sheaf-repr}, there is an order embedding 
\[
\cat{M} \hookrightarrow \sOmShr{A}(L,\cat{C}).
\]

In particular, every map $\Scott(L)\to\RegEpi{A}$ that preserves finite infima and arbitrary suprema, and whose image consists of pairwise ker-commuting elements, induces a soft $\Omega$-sheaf representation of $A$ over $L$. In a similar fashion, if $\cat{C}$ is Barr-exact and there exists a map $\Scott(L)\to \Equiv{A}$ that preserves finite infima and arbitrary suprema, and whose image consists of pairwise commuting equivalence relations, then $A$ admits a soft $\Omega$-sheaf representation over $L$.
\end{remark}

\begin{remark}
Throughout this paper we have adopted an ``external'' perspective, whereby an object of a regular category $\cat{C}$ is studied via (pre)sheaves $L^\op\to \cat{C}$. However it is well known that, in many situations, the appropriate notion of sheaf of $\cat{C}$-objects is given by an ``internal $\cat{C}$-object'' in the category of sheaves of sets. Whereas the internal and external perspectives coincide when $\cat{C}$ is the category of models of a finitary algebraic theory (i.e., $\cat{C}$ is a variety of finitary algebras), they may differ for arbitrary first-order theories---consider e.g.\ the elementary theory of fields.\footnote{This is due to the fact that the global section functor does not preserve the validity of all first-order sentences, but only of \emph{Cartesian theories}, cf.\ e.g.\ \cite[\S V.1.12]{Johnstone1986}.} We do not know if, in general, our results can be adapted to the internal perspective---whenever the latter is available.
\end{remark}

\section{Examples}\label{s:examples}

\subsection{The dual of compact ordered spaces}\label{s:CompOrd}
A \emph{compact ordered space} is a compact space $X$ equipped with a partial order that is closed in the product topology of $X\times X$. Compact ordered spaces were first introduced by Nachbin in his monograph~\cite{Nachbin}.

Let $\CompOrd$ denote the category of compact ordered spaces and continuous monotone maps between them. For any object $X$ of $\CompOrd$ we write $X_*$ for the same object, but this time regarded as an object of the opposite category $\CompOrd^\op$. In this subsection, we shall investigate soft sheaf representations of objects of $\CompOrd^\op$.

Recall that regular monomorphisms in $\CompOrd$ can be identified, up to isomorphism, with the closed subsets with the induced order \cite[Theorem~2.6]{HofmannNevesEtAl2018}.
Moreover, it follows from the main result of~\cite{Abbadini2019} that $\CompOrd^\op$ is a Barr-exact category (see also~\cite{AbbadiniReggio2020} for a direct proof).
Thus, for any compact ordered space $X$, there is an order isomorphism between $(\Equiv{X_*})^\op$ and the coframe of closed subsets of $X$. Equivalently, \[\Equiv{X_*}\cong \Omega(X).\]

For the next lemma recall that, by \cref{cor:QuoA-bounded-lattice} and \cref{rem:Barr-exact-ker-iso}, for any object $A$ of a Barr-exact category admitting pushouts, the poset $\Equiv{A}$ is a (bounded) lattice. Further, given $\theta\in\Equiv{A}$, we write \[A \repi A/{\theta}\] for its coequaliser---provided it exists. If $\theta_1\leq \theta_2$ in $\Equiv{A}$, the universal property of the coequaliser entails the existence of a unique regular epimorphism $A/{\theta_1}\repi A/{\theta_2}$ through which $A \repi A/{\theta_2}$ factors.

\begin{lemma} \label{l:abstract-characterisation}
Let $A$ be an object of a finitely cocomplete Barr-exact category, and let $\theta_1,\theta_2\in\Equiv{A}$. Then $\theta_1$ and $\theta_2$ commute if, and only if, the following diagram (whose arrows are the canonical ones) is a pullback.
	\begin{equation} \label{e:square-pb}
		\begin{tikzcd}
			A/{(\theta_1 \wedge \theta_2)} \arrow[two heads]{r}{} \arrow[two heads]{d}{} & A/{\theta_1} \arrow[two heads]{d}{}\\
			A/{\theta_2} \arrow[two heads]{r}{} & A/{(\theta_1 \vee \theta_2)}
		\end{tikzcd}
	\end{equation}
\end{lemma}

\begin{proof}
The square on the left-hand side below is a pushout by~\cref{l:binary-infima}, where $f$ and $g$ are the coequalisers of $\theta_1$ and $\theta_2$, respectively, and $\eta_1,\eta_2$ are the canonical arrows induced by the universal property of $f$ and $g$, respectively.
\begin{center}
\begin{tikzcd}
		A \arrow[two heads]{r}{f} \arrow[two heads]{d}[swap]{g} & A/{\theta_1} \arrow[two heads]{d}{\eta_1}\\
		A/{\theta_2} \arrow[two heads]{r}{\eta_2} & A/{(\theta_1 \vee \theta_2)} \arrow[ul, phantom, "\ulcorner", very near start]
\end{tikzcd}
\ \ \ \ \ \ \ \ \
\begin{tikzcd}
		A \arrow[bend left = 30, looseness=1, two heads]{rrd}[description]{f} \arrow[swap, bend right = 30, looseness=1, two heads]{ddr}[description]{g}\arrow[dashed]{rd}[description]{q} && \\
		& P \arrow[dr, phantom, "\lrcorner", very near start] \arrow[twoheadrightarrow]{r}{} \arrow[twoheadrightarrow]{d}[swap]{} & A/\theta_1 \arrow[twoheadrightarrow]{d}{\eta_1}\\
		&A/\theta_2 \arrow[twoheadrightarrow]{r}{\eta_2}&A/(\theta_1 \vee \theta_2)
\end{tikzcd}	
\end{center}
Let $P$ be the pullback of $\eta_1$ along $\eta_2$, and let $q \colon A \to P$ denote the unique arrow making the right-hand diagram above commute.
	By~\cref{l:commuting}, $\theta_1$ and $\theta_2$ commute precisely when $q$ is a regular epimorphism.
	
Now, if $q$ is a regular epimorphism then it is an infimum of $f$ and $g$ in $\RegEpi{A}$; so, up to an isomorphism, it coincides with the coequaliser $A \repi A/(\theta_1 \wedge \theta_2)$. Thus, \eqref{e:square-pb} is a pullback square.
Conversely, suppose \eqref{e:square-pb} is a pullback.
Then $q$ is a coequaliser of $\theta_1\wedge \theta_2$, hence a regular epimorphism.
\end{proof}

In \cite[Lemma~5.4]{GehrkeGool2018}, Gehrke and van Gool characterised the pairs of commuting congruences on a bounded distributive lattice $A$ in terms of the corresponding closed subsets of the dual \emph{Priestley space} of $A$~\cite{Priestley1970}. The next result extends their characterisation from Priestley spaces to compact ordered spaces.

\begin{proposition}\label{p:char-commuting}
Let $X$ be a compact ordered space. Let $\theta_1,\theta_2\in \Equiv{X_*}$ and let $C_1,C_2$ be the corresponding closed subsets of $X$. The following statements are equivalent:
	\begin{enumerate}
		\item \label{i:def-commute}
		The equivalence relations $\theta_1$ and $\theta_2$ commute.
		
		\item \label{i:char-commute}
		For any $x_1 \in C_1$, $x_2 \in C_2$, if $\{i,j\} = \{1,2\}$ and $x_i \leq x_j$, there exists $z \in C_1 \cap C_2$ such that $x_i \leq z \leq x_j$.
	\end{enumerate}
\end{proposition}
\begin{proof}
	The category $\CompOrd^\op$ is cocomplete and Barr-exact (cf.~\cite{AbbadiniReggio2020}), thus \cref{l:abstract-characterisation} entails that $\theta_1$ and $\theta_2$ commute if and only if the following is a pushout in $\CompOrd$.
	\begin{equation*}
		\begin{tikzcd}
			C_1 \cap C_2 \arrow[hookrightarrow]{r}{} \arrow[hookrightarrow]{d}{} & C_1 \arrow[hookrightarrow]{d}{} \\
			C_2 \arrow[swap, hookrightarrow]{r}{} & C_1 \cup C_2
		\end{tikzcd}
	\end{equation*}
	In turn, this is equivalent to the condition in \cref{i:char-commute}. Cf.\ e.g.\ \cite[Remark 6]{AbbadiniReggio2020}.
\end{proof}

If $X$ is a compact ordered space, we can consider the collection of all open subsets of $X$ that are downwards closed in the partial order of $X$. The latter forms a topology on the underlying set of $X$, and we shall denote by $X^\downarrow$ the ensuing topological space. In fact, $X^\downarrow$ is a stably compact space (cf.~\cref{ex:stably-compact}) and every stably compact space arises in this manner, cf.\ e.g.\ \cite[Proposition~2.10]{Lawson2011}. Similarly for the space $X^\uparrow$ obtained by considering the topology consisting of the open subsets of $X$ that are upwards closed. The space $X^\uparrow$ is called the \emph{co-compact dual} of $X^\downarrow$ and there is a frame isomorphism $\Omega(X^\downarrow) \cong \K(X^\uparrow)^\op$ sending an open subset of $X^\downarrow$ to its complement, cf.\ e.g.\ \cite[\S 2.2]{Jung2004}.

The following is a direct generalisation of \cite[Definition~5.5]{GehrkeGool2018}, from the setting of Priestley spaces to that of compact ordered spaces.
\begin{definition} \label{d:interpolating}
Let $X, Y$ be compact ordered spaces. An \emph{interpolating decomposition of $X$ over $Y$} is a continuous function $q \colon X \to Y^\downarrow$ such that, for all $x_1, x_2 \in X$, if $x_1 \leq x_2$ then there is $z \in X$ such that $x_1 \leq z \leq x_2$, $q(x_1) \leq q(z)$ and $q(x_2) \leq q(z)$.\footnote{The inequalities $q(x_1) \leq q(z)$ and $q(x_2) \leq q(z)$ refer to the partial order of $Y$.}
\end{definition}

If $X, Y$ are compact ordered spaces and ${q \colon X \to Y^\downarrow}$ is a continuous map, denote by 
\[
\psi_q \colon \Omega(Y^\downarrow) \to \Equiv{X_*}
\] 
the composition of the frame homomorphism $\Omega(q) \colon \Omega(Y^\downarrow) \to \Omega(X)$ with the order isomorphism $\Omega(X) \cong \Equiv{X_*}$. It is useful to consider the map $\phi_q$ order-dual to $\psi_q$. Since $\Omega(Y^\downarrow)^\op \cong \K(Y^\uparrow)$, we can assume that this order-dual map has type
\[
\phi_q \colon \K(Y^\uparrow) \to (\Equiv{X_*})^\op.
\]
Note that any compact ordered space is Hausdorff, hence the poset $\K(X)$ of compact saturated subsets of $X$ coincides with the coframe of its closed subsets. Thus, $\phi_q$ can be equivalently described as the composite of the inverse image map $q^{-1} \colon \K(Y^\uparrow) \to \K(X)$ with the order isomorphism $\K(X) \cong (\Equiv{X_*})^\op$.

\begin{proposition} \label{p:interpolating}
Let $X, Y$ be compact ordered spaces and let ${q \colon X \to Y^\downarrow}$ be a continuous function. The following statements are equivalent:
	\begin{enumerate}
		\item \label{i:decomp}
		The function $q$ is an interpolating decomposition of $X$ over $Y$.
		\item \label{i:commute}
		Any two equivalence relations in the image of $\psi_q \colon \Omega(Y^\downarrow) \to \Equiv{X_*}$ commute.
	\end{enumerate}
\end{proposition}  

\begin{proof}
	We prove, equivalently, that \ref{i:decomp} holds if and only if any two elements in the image of $\phi_q$ (the map order-dual to $\psi_q$) commute.
	
	Assume \ref{i:decomp} holds and let $K_1, K_2 \in \K(Y^\uparrow)$.
	To show that $\phi_q(K_1)$ and $\phi_q(K_2)$ commute, it suffices to prove that the closed sets $C_i \coloneqq q^{-1}(K_i)$, for $i \in \{1,2\}$, satisfy the property in \cref{i:char-commute} of \cref{p:char-commuting}.
	Fix arbitrary elements $x_1\in C_1$ and $x_2 \in C_2$ such that $x_1 \leq x_2$ (if $x_2\leq x_1$ the proof is the same, mutatis mutandis). Because $q$ is an interpolating decomposition, there is $z\in X$ such that $x_1 \leq z \leq x_2$, $q(x_1) \leq q(z)$ and $q(x_2) \leq q(z)$. It remains to show that $z \in C_1 \cap C_2$.
	As $K_1$ and $K_2$ are compact saturated subsets of~$Y^\uparrow$, they are upwards closed in the order of $Y$ (cf.\ e.g.\ \cite[\S 2.2]{Jung2004}).
	Since $q(x_i) \in K_i$ for $i \in \{1,2\}$, we get $q(z) \in K_1 \cap K_2$ and so $z \in C_1 \cap C_2$.
		
	Conversely, suppose any two elements in the image of $\phi_q$ commute and let $x_1, x_2 \in X$ satisfy $x_1 \leq x_2$. Write $y_i \coloneqq q(x_i)$ and $C_i \coloneqq q^{-1}(\u y_i$) for $i \in \{1,2\}$. Clearly, $x_i \in C_i$ for $i \in \{1,2\}$. Moreover, by definition of $\phi_q$, $C_i$ is the closed subset of $X$ corresponding to the equivalence relation $\phi_q(\u y_i)$ on~$X_*$. Since $\phi_q(\u y_1)$ and $\phi_q(\u y_2)$ commute, by \cref{p:char-commuting} there is $z \in C_1 \cap C_2$ such that $x_1 \leq z \leq x_2$.
	Note that, for each $i \in \{1,2\}$, $z\in C_i$ implies that $q(x_i) \leq q(z)$.
\end{proof}

Let $X,Y$ be compact ordered spaces and let $(F, \phi)$ be a soft $\K$-sheaf representation of $X_*$ over the coframe $\K(Y^\uparrow)$. 
Under the order isomorphism in \cref{cor:MAIN}, the isomorphism class of $(F, \phi)$ corresponds to a frame homomorphism ${\K(Y^\uparrow)^\op \to \Equiv{X_*}}$. Identifying $\Equiv{X_*}$ with $\Omega(X)$, and $\K(Y^\uparrow)^\op$ with $\Omega(Y^\downarrow)$, we shall denote this frame homomorphism by $\psi_{(F, \phi)} \colon \Omega(Y^\downarrow) \to \Omega(X)$. Note that $Y^\downarrow$ and $X$ are sober spaces, so there is a unique continuous function 
\[
q_{(F, \phi)} \colon X \to Y^\downarrow
\] 
such that $\Omega(q_{(F, \phi)}) = \psi_{(F, \phi)}$.
The following theorem extends Gehrke and van Gool's result \cite[Theorem~5.7]{GehrkeGool2018} from distributive lattices to the dual of compact ordered spaces.

\begin{theorem}\label{thm:sheaf-repr-interp-bijection}
	Let $X, Y$ be compact ordered spaces. The assignment \[(F, \phi) \mapsto q_{(F, \phi)}\] yields a bijection between isomorphism classes of soft $\K$-sheaf representations of $X_*$ over $\K(Y^\uparrow)$ and interpolating decompositions of $X$ over $Y$.
\end{theorem}
\begin{proof}
	This follows from \cref{cor:MAIN} and \cref{p:interpolating}.
\end{proof}

Observe that, for every compact ordered space $X$, there is a fully faithful functor 
\[
\sKShr{X_*}(\K(Y^\uparrow),\CompOrd^\op) \hookrightarrow \sOmShr{X_*}(\Omega(Y^\uparrow),\CompOrd^\op)
\]
from soft $\K$-sheaf representations of $X_*$ over $\K(Y^\uparrow)$ to soft sheaf representations of $X_*$ over $Y^\uparrow$, and this is an equivalence provided that directed colimits commute with finite limits in $\CompOrd^\op$ (cf.~\cref{pr:inclusion-K-into-Omega} and \cref{equiv-Om-K-sheaves}). We do not know if the category $\CompOrd^\op$ satisfies the latter property and conjecture that it does not. 

Nevertheless, replacing $\K$-sheaves over $\K(Y^\uparrow)$ with ordinary sheaves over $Y^\uparrow$ in \cref{thm:sheaf-repr-interp-bijection}, we obtain an injective assignment from interpolating decompositions of $X$ over $Y$ into isomorphism classes of soft sheaf representations of $X_*$ over $Y^\uparrow$ (cf.\ \cref{rem:M-into-Omega-shvs}).
This allows us to construct soft sheaf representations for all objects of $\CompOrd^\op$:
\begin{proposition} 
	Let $X$ be a compact ordered space. Then $X_*$ admits a soft sheaf representation over the stably compact space $X^\uparrow$ induced by the interpolating decomposition $X\to X^\downarrow$ given by the identity function.
\end{proposition} 
Of course, the category $\CompOrd^\op$ can be replaced with any equivalent category $\cat{D}$ (e.g., following~\cite{Abbadini2019}, with an appropriate variety of infinitary algebras; cf.\ also~\cite{HofmannNevesEtAl2018}), thus obtaining soft sheaf representations of objects of $\cat{D}$.

\subsection{Commutative Gelfand rings}\label{s:comm-Gelfand-rings}
In this subsection, we shall assume that the reader is familiar with basic notions of point-free topology, see e.g.\ \cite{PP2012}.

Let $\CRing$ be the category of commutative rings with unit and ring homomorphisms preserving the unit, and fix an arbitrary $A\in\CRing$.
Denote by $X_A$ the Zariski spectrum of $A$ (that is, $X_A$ is the set of prime ideals of $A$ equipped with the Zariski or \emph{hull-kernel} topology). A classical result by Grothendieck~\cite{EGA1} states that $A$ is isomorphic to the ring of global sections of a sheaf $F\colon \Omega(X_A)^\op\to\CRing$ whose stalks are local rings. In fact, the stalk of $F$ at $\mathfrak{p}\in X_A$ is isomorphic to the localization $A_{\mathfrak{p}}$ of $A$ at~$\mathfrak{p}$. Note that this is \emph{not} a soft sheaf representation because the stalks of $F$ are not quotients of $A$ (cf.\ \cref{rem:soft-sh-repr-stalks-as-quotients}). In more detail, for all $\mathfrak{p}\in X_A$, the set 
\[
k_{\mathfrak{p}}\coloneqq \{U\in \Omega(X_A)\mid \mathfrak{p}\in U\}
\]
is a filter on $\Omega(X_A)$ and is Scott-open because $X_A$ is locally compact. Thus, the canonical colimit arrow $F(X_A)\to \colim_{U\in k_{\mathfrak{p}}}{F(U)}$ into the stalk of $F$ at $\mathfrak{p}$ can be identified with the localization map
\[
A \to A_{\mathfrak{p}}.
\]
The latter is an epimorphism but, in general, fails to be surjective (i.e., a regular epimorphism in the category $\CRing$).
\begin{remark}
	On the other hand, the \emph{Pierce representation} of a commutative ring $A$ \emph{is} a soft sheaf representation. In fact, it is induced by the monotone map \[\ID{E(A)} \to \ID{A},\] where $E(A)$ is the Boolean ring of idempotent elements of $A$, that sends an ideal $J$ of $E(A)$ to the ideal of $A$ generated by $J$. See e.g.\ \cite[\S V.2]{Johnstone1986}. In this case, $\ID{E(A)}$ can be identified with the frame of opens of the Boolean (i.e., compact, Hausdorff and zero-dimensional) space corresponding to $E(A)$ under Stone duality for Boolean algebras~\cite{Stone1936}. More generally, if $\cat{V}$ is a variety of finitary algebras whose signature contains a constant symbol, \emph{every} sheaf of $\cat{V}$-algebras over a Boolean space is soft, cf.\ \cite[Lemma~3.3]{Pierce1967}.
\end{remark}

However, Grothendieck's sheaf representation induces a soft sheaf representation for a smaller class of commutative rings as we shall now explain. For any $A\in\CRing$, the frame $\RID{A}$ of \emph{radical} ideals\footnote{An ideal is \emph{radical} if it is the intersection of all prime ideals in which it is contained.} of~$A$ (ordered by inclusion) is compact and \emph{coherent}, i.e.\ the subset of $\RID{A}$ consisting of the compact elements forms a join-dense sublattice. 
Assuming the Prime Ideal Theorem, it can be proved that $\RID{A}$ is a spatial frame isomorphic to $\Omega(X_A)$. This observation was exploited by Banaschewski~\cite{Banaschewki2004} to give a point-free version of Grothendieck's sheaf representation of $A$, replacing the spatial frame $\Omega(X_A)$ with~$\RID{A}$. 

Now, recall that a \emph{(commutative) Gelfand ring} is a (commutative) ring with unit satisfying the following condition:
\[
\forall x,y. \, (x+y =1 \ \Longrightarrow \ \exists a,b. \, (1+ xa) (1 + yb) = 1).
\]
Again assuming the Prime Ideal Theorem, commutative Gelfand rings are exactly the commutative rings with unit in which every prime ideal is contained in a unique maximal ideal.
Even in the absence of the Prime Ideal Theorem, we have that a commutative ring $A$ with unit is a commutative Gelfand ring precisely when the frame $\RID{A}$ is normal \cite[Proposition~1]{Banaschewki2004} (recall that a frame $L$ is \emph{normal} if, for all $g,h \in L$ such that $g \lor h = \top$, there are $u,v \in L$ such that $u \lor g = v \lor h = \top$ and $u \wedge v = \bot$). In that case, $\RID{A}$ retracts onto its compact regular subframe $\JRID{A}$ consisting of the \emph{Jacobson radical ideals}, i.e.\ those ideals $J$ such that, for every $a \in A$, if $1+ra$ is invertible modulo $J$ for all $r \in A$, then $a \in J$. See \cite[Lemma~1 and p.~27]{Banaschewki2004}. Since $\JRID{A}$ is a compact regular frame, in view of the following remark there is an isomorphism $\Scott(\JRID{A}) \cong \JRID{A}$. 

\begin{remark}\label{rem:Lawson-dual-of-kreg-frame}
	If $L$ is a compact regular frame, then $\Scott(L)\cong L$. An explicit isomorphism is given by
	\[
	L \to \Scott(L), \ \ x\mapsto \{y\in L\mid x\vee y = 1\},
	\]
	whose inverse sends $k\in\Scott(L)$ to $\sup \{x \in L \mid \exists y \in k. \ x \land y = 0\}$.
	In the particular case of spatial compact regular frames (recall that, assuming the Axiom of Choice, every compact regular frame is spatial, see e.g.\ \cite[Proposition~III.1.10]{Johnstone1986}), this reduces to the observation that a compact Hausdorff topology coincides with its \emph{patch topology}.
\end{remark}

The category $\CRing$ is a well-powered, bicomplete Barr-exact category in which directed colimits commute with finite limits. 
Hence, by \cref{cor:pos-refl-omega} and \cref{rem:well-powered-sheaf-rep}, there is an order isomorphism of small posets
\[
\cat{N}\cong\llbracket \sOmShr{A}(\JRID{A},\CRing) \rrbracket
\]
where $\cat{N}$ consists of the maps $\JRID{A}\to\ID{A}$ preserving finite infima and arbitrary suprema (just observe that any two congruences on~$A$ commute). Because $\CRing$ has no proper subterminal objects, it follows from \cref{rem:Omega-sheaves-vs-sheaves-caveat} that $\CRing$-valued (soft) $\Omega$-sheaves coincide with ordinary (soft) sheaves. Therefore the inclusion
\[
\JRID{A} \hookrightarrow \ID{A},
\]
which preserves finite infima and arbitrary suprema because so does $\RID{A}\hookrightarrow \ID{A}$, induces a soft sheaf representation of the commutative Gelfand ring $A$ over the compact regular frame $\JRID{A}$. This sheaf representation was first obtained by Banaschewski and Vermeulen~\cite{BV2011}, improving on results of Mulvey~\cite{Mulvey1979} and \mbox{Banaschewski~\cite{Banaschewksi2000,Banaschewki2004}.}

If the Prime Ideal Theorem is assumed, for any commutative Gelfand ring $A$ the frame $\JRID{A}$ can be identified with $\Omega(\Max{A})$, where $\Max{A}$ is the subspace of $X_A$ consisting of the maximal ideals of $A$ (i.e., the closed points of $X_A$). The space $\Max{A}$ is compact and Hausdorff, and in view of the previous paragraph $A$ is isomorphic to the ring of global sections of a soft sheaf $F\colon \Omega(\Max{A})^\op\to\CRing$ (this sheaf representation can also be derived as a special case of~\cite[Corollary~3.11]{GehrkeGool2018}, see p.~2178 in \emph{op.\ cit.}). 

Note that, in contrast with the case of arbitrary commutative rings, for Gelfand rings we get a \emph{soft} sheaf representation. In fact, the stalk of $F$ at a maximal ideal $\mathfrak{m}\in\Max{A}$ is isomorphic to the quotient ring $A/O_{\mathfrak{m}}$, where the ideal $O_{\mathfrak{m}}$ is defined by 
\[
O_{\mathfrak{m}}\coloneqq \{a\in A\mid \exists b\in A\setminus \mathfrak{m} \ \text{such that} \ ab=0\}.
\]
See e.g.\ \cite[Lemma~V.3.8]{Johnstone1986}. The unique maximal ideal of $A$ containing $O_{\mathfrak{m}}$ is $\mathfrak{m}$, hence $A/O_{\mathfrak{m}}$ is a local ring (equivalently, note that $A/O_{\mathfrak{m}}\cong A_{\mathfrak{m}}$). The canonical colimit arrow $F(\Max{A})\to \colim_{x\in U}{F(U)}$ can then be identified with the quotient map $A\to A/O_{\mathfrak{m}}$, which is a regular epimorphism. This shows that every local section over a point of $\Max{A}$ can be extended to a global section. A similar argument shows that every local section defined on a closed subset of $\Max{A}$ can be extended to a global section, i.e.\ $F$ is soft.

\subsection*{Acknowledgements}
We are grateful to Pino Rosolini for a number of valuable comments on a preliminary version of this work, and to the anonymous referee for further suggestions that allowed us to improve the presentation of our results.

\bibliographystyle{plain}

\end{document}